\documentclass[final,onefignum,onetabnum]{siamart250211}

\usepackage{amsmath,amsfonts,amssymb,array,graphicx,mathtools,multirow,bm,tcolorbox,relsize,booktabs,dsfont}
\usepackage{mathrsfs}
\usepackage{multirow}
\usepackage{arydshln} 
\allowdisplaybreaks
\usepackage{stmaryrd} 
\usepackage{lipsum}
\usepackage{amsfonts}
\usepackage{graphicx}
\usepackage{epstopdf}
\usepackage{algorithmic}

\ifpdf
  \DeclareGraphicsExtensions{.eps,.pdf,.png,.jpg}
\else
  \DeclareGraphicsExtensions{.eps}
\fi

\newsiamremark{remark}{Remark}
\newsiamremark{hypothesis}{Hypothesis}
\crefname{hypothesis}{Hypothesis}{Hypotheses}
\newsiamthm{claim}{Claim}

\usepackage{tikz}
\usepackage{tikz-3dplot}
\usetikzlibrary{patterns}
\usetikzlibrary{shapes,calc,shapes,arrows}

\headers{Quotient geometry of tensor ring decomposition}{
B. Gao, R. Peng, and Y.-x. Yuan}

\title{Quotient geometry of tensor ring decomposition\thanks{Submitted to the editors DATE.
\funding{This work was supported by the National Key R\&D Program of China (grant 2023YFA1009300). BG and YY were supported by the National Natural Science Foundation of China (grant No.~12288201).}}}

\author{
    Bin Gao\thanks{State Key Laboratory of Mathematical Sciences, Academy of Mathematics and Systems Science, Chinese  Academy of Sciences, China
		({\{gaobin,yyx\}@lsec.cc.ac.cn}).}
	\and Renfeng Peng\thanks{State Key Laboratory of Mathematical Sciences, Academy of Mathematics and Systems Science, Chinese  Academy of Sciences, and University of Chinese Academy of Sciences, China ({pengrenfeng@lsec.cc.ac.cn}).}
	\and Ya-xiang Yuan\footnotemark[2]
}

\usepackage{amsopn}

\usepackage{amssymb}

\newcommand{\vecb}{\mathbf{b}}

\newcommand{\vecx}{\mathbf{x}}
\newcommand{\vecr}{\mathbf{r}}

\newcommand{\vecmatA}{\vec{\mathbf{A}}}

\newcommand{\mata}{\mathbf{A}}
\newcommand{\matb}{\mathbf{B}}

\newcommand{\matd}{\mathbf{D}}
\newcommand{\mate}{\mathbf{E}}

\newcommand{\matp}{\mathbf{P}}
\newcommand{\matx}{\mathbf{X}}

\newcommand{\matu}{\mathbf{U}}
\newcommand{\matv}{\mathbf{V}}
\newcommand{\matW}{\mathbf{W}}

\newcommand{\matA}{\mathbf{A}}

\newcommand{\matD}{\mathbf{D}}

\newcommand{\matI}{\mathbf{I}}

\newcommand{\matR}{\mathbf{R}}

\newcommand{\matX}{\mathbf{X}}
\newcommand{\matU}{\mathbf{U}}

\newcommand{\tensA}{\mathcal{A}}

\newcommand{\tensE}{\mathcal{E}}
\newcommand{\tensF}{\mathcal{F}}
\newcommand{\tensG}{\mathcal{G}}

\newcommand{\tensM}{\mathcal{M}}

\newcommand{\tensR}{\mathcal{R}}

\newcommand{\tensT}{\mathcal{T}}
\newcommand{\tensU}{\mathcal{U}}
\newcommand{\tensV}{\mathcal{V}}
\newcommand{\tensY}{\mathcal{Y}}
\newcommand{\tensX}{\mathcal{X}}

\newcommand{\rank}{\mathrm{rank}}
\newcommand{\reshape}{\mathrm{reshape}}

\newcommand{\ranktt}{\mathrm{rank}_{\mathrm{TT}}}
\newcommand{\ranktr}{\mathrm{rank}_{\mathrm{TR}}}
\newcommand{\rmvec}{\mathrm{vec}}
\newcommand{\tr}{\mathrm{tr}}

\newcommand{\Span}{\mathrm{span}}
\newcommand{\grad}{\mathrm{grad}}
\newcommand{\lift}{\mathrm{lift}}

\newcommand{\subjectto}{\mathrm{s.\,t.}}

\newcommand{\GL}{\mathrm{GL}}
\newcommand{\PGL}{\mathrm{PGL}}

\newcommand{\vectensU}{{\vec{\tensU}}}
\newcommand{\vectensV}{{\vec{\tensV}}}

\newcommand{\barf}{{\bar{f}}}
\newcommand{\barg}{{\bar{g}}}

\newcommand{\bareta}{{\bar{\eta}}}

\newcommand{\bartensM}{{\overline{\tensM}}}

\newcommand{\frob}{\mathrm{F}}

\usepackage{tikz}
\usepackage{tikz-3dplot}
\usetikzlibrary{patterns}
\usetikzlibrary{shapes,calc,shapes,arrows}

\makeatletter
\@addtoreset{equation}{section}
\makeatother

\DeclareMathOperator{\diag}{diag}

\DeclareMathOperator{\tangent}{T}
\DeclareMathOperator{\vertical}{V}
\DeclareMathOperator{\horizontal}{H}

\DeclareMathOperator{\retr}{R}

\DeclareMathOperator{\proj}{P}

\begin{document}

\maketitle

\begin{abstract}
    Differential geometries derived from tensor decompositions have been extensively studied and provided the foundations for a variety of efficient numerical methods. Despite the practical success of the tensor ring (TR) decomposition, its intrinsic geometry remains less understood, primarily due to the underlying ring structure and the resulting nontrivial gauge invariance. We establish the quotient geometry {{\color{black}and immersed-submanifold structure}} of TR decomposition by imposing full-rank conditions on all unfolding matrices of the core tensors and capturing the gauge invariance. {{\color{black}The intrinsic ring structure of TR leads to an analysis that is substantially different from other tensor formats.}}
	Additionally, {{\color{black}for the uniform TR decomposition, where all core tensors are identical and the manifold structure is known, we derive explicit parameterizations for the vertical and horizontal spaces, which enable Riemannian optimization.}} Numerical experiments validate the developed geometries via tensor ring completion tasks.
\end{abstract}



\section{Introduction}
Tensors, or multidimensional arrays, are ubiquitous for representing high-dimensional data. However, storing a tensor in full size suffers from the curse of dimensionality, as the required memory grows exponentially with the order of a tensor. Low-rank tensor decompositions not only save storage, but also extract the most essential information from a tensor. Low-rank tensors have demonstrated effectiveness in various applications, including image processing~\cite{vasilescu2003multilinear}, tensor completion~\cite{kressner2014low,gao2025low}, high-dimensional eigenvalue problems~\cite{dolgov2014computation}, and high-dimensional partial differential equations~\cite{bachmayr2023low}; see~\cite{grasedyck2013literature} for an overview. In contrast with low-rank matrices, the low-rank structure of a tensor depends on a specific tensor decomposition format. The canonical polyadic decomposition~\cite{hitchcock1928multiple}, Tucker decomposition~\cite{de2000multilinear}, hierarchical Tucker decomposition~\cite{uschmajew2013geometry}, tensor train (TT) decomposition~\cite{oseledets2011tensor} and tensor ring (TR) decomposition~\cite{zhao2016tensor} (also known as matrix product states (MPS) in quantum physics~\cite{verstraete2004density} with open and periodic boundary conditions respectively) are among the most typical formats. We refer to~\cite{kolda2009tensor} for an overview. 

In this paper, we consider the tensor ring decomposition, which decomposes a $d$-th order tensor $\tensX\in\mathbb{R}^{n_1\times n_2\times\cdots\times n_d}$ into $d$ core tensors $\tensU_k\in\mathbb{R}^{r_k\times n_k\times r_{k+1}}$ for $k=1,2,\dots,d$ with $r_{d+1}=r_1$, denoted by $\tensX=\llbracket\tensU_1,\tensU_2,\dots,\tensU_d\rrbracket$. The $(i_1,i_2,\dots,i_d)$-th entry of $\tensX$ is represented by the trace of products of~$d$ matrices, i.e., 
\begin{equation}
	\label{eq: elementwise TR}
	\tensX=\llbracket\tensU_1,\tensU_2,\dots,\tensU_d\rrbracket\quad\text{and}\quad\tensX(i_1,i_2,\dots,i_d)=\tr\left(\matU_1(i_1)\matU_2(i_2)\cdots\matU_d(i_d)\right),
\end{equation}
where $\matu_k(i_k)=\tensU_k(:,i_k,:)\in\mathbb{R}^{r_{k}\times r_{k+1}}$ is the lateral slice matrix of $\tensU_k$ for $i_k=1,2,\dots,n_k$. \cref{fig: Tensor Ring} (left) depicts the TR decomposition of a tensor via tensor networks. Note that TR boils down to the TT decomposition in~\cref{fig: Tensor Ring} (middle) if $r_1=r_{d+1}=1$, and to the uniform TR decomposition (also known as the translation-invariant MPS in physics~\cite{haegeman2014geometry}) in~\cref{fig: Tensor Ring} (right) if all core tensors are identical, i.e., $\tensU_1=\tensU_2=\cdots=\tensU_d$.

\begin{figure}[htbp]
	\centering
	\begin{tikzpicture}\scriptsize
		\coordinate (O) at (0,0); 
        \def\l{0.9};

        \draw[thin,dashed,rounded corners] ($(O)-(\l+0.2,1.732*\l+0.2)$) rectangle ($(O)-(\l+0.2,1.732*\l+0.2)+(4*\l+0.4,2*1.732*\l+0.4)$);
        \node[above] at ($(O)-(\l+0.2,1.732*\l+0.2)+(2*\l+0.2,2*1.732*\l+0.4)$) {\normalsize Tensor ring};

		\node[circle,draw,inner sep=1pt] (U1) at (O) {\footnotesize$\tensU_1$};
		\node[circle,draw,inner sep=1pt] (U2) at ($(O)+(0.5*\l,-0.866*\l)$) {\footnotesize$\tensU_2$};
		\node[circle,draw,inner sep=1pt] (U3) at ($(O)+(1.5*\l,-0.866*\l)$) {\footnotesize$\tensU_3$};
		\node[circle,draw,inner sep=1pt] (U4) at ($(O)+(2*\l,0)$) {\footnotesize$\tensU_4$};
		\node[circle,draw,inner sep=1pt] (U5) at ($(O)+(1.5*\l,0.866*\l)$) {\footnotesize$\tensU_5$};
		\node[circle,draw,inner sep=1pt] (U6) at ($(O)+(0.5*\l,0.866*\l)$) {\footnotesize$\tensU_{d}$};
        
        \draw[-] (U6)--(U1)--(U2)--(U3)--(U4)-- (U5);
        \draw[dashed] (U5)--(U6);
        \draw (U1) -- ($(U1)+(-\l,0)$);
        \draw (U2) -- ($(U2)+(-0.5*\l,-0.866*\l)$);
        \draw (U3) -- ($(U3)+(0.5*\l,-0.866*\l)$);
        \draw (U4) -- ($(U4)+(\l,0)$);
        \draw (U5) -- ($(U5)+(0.5*\l,0.866*\l)$);
        \draw (U6) -- ($(U6)+(-0.5*\l,0.866*\l)$);
        
        \node at ($0.5*(U1)+0.5*(U2)+0.2*(-0.866*\l,-0.5*\l)$) {$r_2$};
        \node[below] at ($0.5*(U2)+0.5*(U3)$) {$r_3$};
        \node at ($0.5*(U3)+0.5*(U4)+0.2*(0.866*\l,-0.5*\l)$) {$r_4$};
        \node at ($0.5*(U4)+0.5*(U5)+0.2*(0.866*\l,0.5*\l)$) {$r_5$};
        \node at ($0.5*(U6)+0.5*(U1)+0.2*(-0.866*\l,0.5*\l)$) {$r_1$};
        
        \node[above] at ($(U1)+0.6*(-\l,0)$) {$n_1$};
        \node at ($(U2)+0.6*(-0.5*\l,-0.866*\l)+0.2*(-0.866*\l,0.5*\l)$) {$n_2$};
        \node at ($(U3)+0.6*(0.5*\l,-0.866*\l)+0.2*(-0.866*\l,-0.5*\l)$) {$n_3$};
        \node[below] at ($(U4)+0.6*(\l,0)$) {$n_4$};
        \node at ($(U5)+0.6*(0.5*\l,0.866*\l)+0.2*(0.866*\l,-0.5*\l)$) {$n_5$};
        \node at ($(U6)+0.6*(-0.5*\l,0.866*\l)+0.2*(0.866*\l,0.5*\l)$) {$n_d$};

        \coordinate (Ott) at (3.8,1.732*\l-0.3-0.8*\l); 
        \def\l{1.1};
        \scriptsize

        \draw[thin,dashed,rounded corners] ($(Ott)-(0.5,\l+0.5)$) rectangle ($(Ott)-(0.5,0)+(3*\l+1,0.5)$);
        \node[above] at ($(Ott)-(0.5,0)+(1.5*\l+0.5,0.5)$) {\normalsize Tensor train};
        
		\node[circle,draw,inner sep=1pt] (U1) at (Ott) {\footnotesize$\tensU_1$};
		\node[circle,draw,inner sep=1pt] (U2) at ($(Ott)+(\l,0)$) {\footnotesize$\tensU_2$};
		\node[inner sep=1pt] (U3) at ($(Ott)+2*(\l,0)$) {\footnotesize$\cdots$};
		\node[circle,draw,inner sep=1pt] (U4) at ($(Ott)+3*(\l,0)$) {\footnotesize$\tensU_d$};
        
        \draw[-] (U1)--(U2)--(U3)--(U4);
        \node[above] at ($0.5*(U1)+0.5*(U2)$) {$r_2$};
        \node[above] at ($0.5*(U2)+0.5*(U3)$) {$r_3$};
        \node[above] at ($0.5*(U3)+0.5*(U4)$) {$r_{d}$};

        \draw[-] (U1) -- ($(U1)+(0,-\l)$);
        \draw[-] (U2) -- ($(U2)+(0,-\l)$);
        \draw[-] (U4) -- ($(U4)+(0,-\l)$);

        \node[left] at ($(U1)+0.6*(0,-\l)$) {$n_1$};
        \node[left] at ($(U2)+0.6*(0,-\l)$) {$n_2$};
        \node[left] at ($(U4)+0.6*(0,-\l)$) {$n_d$};

        \coordinate (Outr) at (9,0); 
        \scriptsize
        \def\l{0.8};

        \draw[thin,dashed,rounded corners] ($(Outr)-(\l+0.2,1.732*\l+0.2)$) rectangle ($(Outr)-(\l+0.2,1.732*\l+0.2)+(4*\l+0.4,2*1.732*\l+0.4)$);
        \node[above] at ($(Outr)-(\l+0.2,1.732*\l+0.2)+(2*\l+0.2,2*1.732*\l+0.4)$) {\normalsize Uniform tensor ring};
        
		\node[circle,draw,inner sep=2pt,fill=gray!20] (U1) at (Outr) {\footnotesize$\tensU$};
		\node[circle,draw,inner sep=2pt,fill=gray!20] (U2) at ($(Outr)+(0.5*\l,-0.866*\l)$) {\footnotesize$\tensU$};
		\node[circle,draw,inner sep=2pt,fill=gray!20] (U3) at ($(Outr)+(1.5*\l,-0.866*\l)$) {\footnotesize$\tensU$};
		\node[circle,draw,inner sep=2pt,fill=gray!20] (U4) at ($(Outr)+(2*\l,0)$) {$\tensU$};
		\node[circle,draw,inner sep=2pt,fill=gray!20] (U5) at ($(Outr)+(1.5*\l,0.866*\l)$) {\footnotesize$\tensU$};
		\node[circle,draw,inner sep=2pt,fill=gray!20] (U6) at ($(Outr)+(0.5*\l,0.866*\l)$) {\footnotesize$\tensU$};
        
        \draw[-] (U6)--(U1)--(U2)--(U3)--(U4)--(U5);
        \draw[dashed] (U5)--(U6);
        \draw (U1) -- ($(U1)+(-\l,0)$);
        \draw (U2) -- ($(U2)+(-0.5*\l,-0.866*\l)$);
        \draw (U3) -- ($(U3)+(0.5*\l,-0.866*\l)$);
        \draw (U4) -- ($(U4)+(\l,0)$);
        \draw (U5) -- ($(U5)+(0.5*\l,0.866*\l)$);
        \draw (U6) -- ($(U6)+(-0.5*\l,0.866*\l)$);
        
        \node at ($0.5*(U1)+0.5*(U2)+0.2*(-0.866*\l,-0.5*\l)$) {$r$};
        \node[below] at ($0.5*(U2)+0.5*(U3)$) {$r$};
        \node at ($0.5*(U3)+0.5*(U4)+0.2*(0.866*\l,-0.5*\l)$) {$r$};
        \node at ($0.5*(U4)+0.5*(U5)+0.2*(0.866*\l,0.5*\l)$) {$r$};
        \node at ($0.5*(U6)+0.5*(U1)+0.2*(-0.866*\l,0.5*\l)$) {$r$};
        
        \node[above] at ($(U1)+0.6*(-\l,0)$) {$n$};
        \node at ($(U2)+0.6*(-0.5*\l,-0.866*\l)+0.2*(-0.866*\l,0.5*\l)$) {$n$};
        \node at ($(U3)+0.6*(0.5*\l,-0.866*\l)+0.2*(-0.866*\l,-0.5*\l)$) {$n$};
        \node[below] at ($(U4)+0.6*(\l,0)$) {$n$};
        \node at ($(U5)+0.6*(0.5*\l,0.866*\l)+0.2*(0.866*\l,-0.5*\l)$) {$n$};
        \node at ($(U6)+0.6*(-0.5*\l,0.866*\l)+0.2*(0.866*\l,0.5*\l)$) {$n$};

	\end{tikzpicture}
	\caption{Illustration of tensor ring decomposition, tensor train decomposition, and uniform tensor ring decomposition of a tensor via tensor networks.}
	\label{fig: Tensor Ring}
\end{figure}

TR decomposition can be viewed as a generalization of TT decomposition, and has several advantages over TT. On the one hand, TR decomposition permits a more comprehensive exploration of information along mode-$1$ and mode-$d$ by relaxing the rank constraint on the first and last cores in TT from one to an arbitrary integer, i.e., from TT rank $(1,r_1,r_2,\dots,r_{d-1},1)$ to TR rank $(r_1,r_2,\dots,r_d)$, enabling higher compressibility and flexibility~\cite{zhao2019learning}. On the other hand, tensor ring decomposition produces reliable approximations to ground states of local Hamiltonians with modest parameter $\vecr=(r_1,r_2,\dots,r_d)$~\cite{pippan2010efficient}. These properties highlight the expressive power of the TR format. 


\paragraph{Related work and motivation}
We provide an overview of the existing tensor decompositions and the related differential geometries. 
The rank of a tensor depends on a specific tensor decomposition format. 
{{\color{black}For tensors of order at least three, the set of bounded-rank tensors in CP format are not necessarily closed. As a result, the best low-rank approximation in CP format may fail to exist~\cite{de2008tensor}. Alternatively, one can alternatively study the geometry of tensors in CP format through a product manifold of rank-one tensors or factor matrices, or through a homogeneous perspective; see, e.g.,~\cite{swijsen2022tensor,dong2022new,jacobsson2025homogeneous}.}}
For Tucker decomposition, Koch and Lubich~\cite{koch2010dynamical} derived a representation of tangent spaces of the set of fixed-rank tensors. Kasai and Mishra~\cite{kasai2016low} studied the quotient geometry derived from the parameter space of Tucker decomposition. Uschmajew and Vandereycken~\cite{uschmajew2013geometry} developed the embedded and quotient geometries of tensors in hierarchical Tucker format. Holtz et al.~\cite{holtz2012manifolds} showed that the set of fixed rank tensors in the TT format forms a smooth manifold. Cai et al.~\cite{cai2026tensor} introduced a preconditioned metric on the quotient manifold of tensor train decomposition. Recently, a desingularization approach was proposed in~\cite{gao2024desingularization} for bounded-rank tensors in the Tucker and TT format, which results in a smooth manifold. For tensor ring decomposition, a non-Euclidean metric was developed~\cite{gao2024riemannian} on the parameter space
\[\bartensM_\vecr:=\mathbb{R}^{r_1\times n_1\times r_2}\times \mathbb{R}^{r_2\times n_2\times r_3}\times\cdots\times\mathbb{R}^{r_d\times n_d\times r_1}\] 
to accelerate Riemannian methods. We refer to~\cite{uschmajew2020geometric} for an overview of geometric methods on low-rank tensors.

Unlike the well-explored manifolds of fixed-rank tensors in the (hierarchical) Tucker and TT formats, the existing work on the geometry of tensors in tensor ring format is quite sparse. {{\color{black}The intrinsic geometry derived from TR decomposition---providing insights into the nature of the decomposition and forming the basis for developing efficient numerical methods---remains less understood.}} First, in contrast with the TT decomposition, the relaxed rank constraint in TR introduces additional degrees of freedom, which leads to the nonuniqueness of the TR rank~\cite[\S 9]{ye2018tensor}. Second, the ring structure renders a different \emph{gauge invariance}. Specifically, an additional matrix $\mata\mata^{-1}$ can be freely inserted between mode~$d$ and mode~$1$ (see~\cref{fig:gauge_inv}), which complicates the definition of tangent spaces, projection operators, and other fundamental geometric tools. Therefore, new strategies are required to address the gauge invariance, and thus, the manifold structure for TR cannot be straightforwardly generalized from TT. 

\paragraph{Contributions}
We establish the quotient geometry of tensor ring decomposition. First, we impose full-rank conditions on the core tensors, i.e., the mode-2 unfolding matrix $(\tensU_k)_{(2)}$ satisfies $\rank((\tensU_k)_{(2)})=r_{k}r_{k+1}$ for $k=1,2,\dots,d$, and we yield a new parameter space
\[\bartensM_\vecr^*:=\mathbb{R}^{r_1\times n_1\times r_2}_*\times\mathbb{R}^{r_2\times n_2\times r_3}_*\times\cdots\times\mathbb{R}^{r_{d}\times n_d\times r_{1}}_*.\]
Observe that inserting any matrix $\mata_k^{-1}\mata_k^{}$ between the core tensors $\tensU_{k-1}$ and $\tensU_k$ in~\cref{fig: Tensor Ring} (left) leaves the represented tensor unchanged, which is referred to as the gauge invariance. Therefore, we can introduce an equivalence class on the new parameter space $\bartensM_\vecr^*$ through the gauge invariance: 
\[(\tensU_1,\tensU_2,\dots,\tensU_d)\sim(\tensV_1,\tensV_2,\dots,\tensV_d)\quad\iff\quad\llbracket\tensU_1,\tensU_2,\dots,\tensU_d\rrbracket=\llbracket\tensV_1,\tensV_2,\dots,\tensV_d\rrbracket,\]
or equivalently, there exists $(\mata_1,\mata_2,\dots,\mata_d)\in\GL(r_1)\times\GL(r_2)\times\cdots\GL(r_d)$ such that $\tensU_k=\tensV_k\times_1\mata_k\times_3\mata_{k+1}^{-\top}$ for $k=1,2,\dots,d$; see~\cref{fig:gauge_inv} for an illustration. Note that the ring structure induces an intrinsic global scaling ambiguity, i.e., $\tensV_k\times_1\mata_k\times_3\mata_{k+1}^{-\top}=\tensV_k\times_1(c\mata_k)\times_3(c\mata_{k+1})^{-\top}$ for any $c\in\mathbb{R}_*:=\mathbb{R}\setminus\{0\}$. Therefore, we consider the projective Lie group $\PGL(\vecr)=\GL(r_1)\times\GL(r_2)\times\cdots\times\GL(r_d)/\mathbb{R}_*$ and introduce the quotient set 
\[\tensM_\vecr^*=\bartensM_\vecr^*/\PGL(\vecr)\]
via the equivalence $\sim$. We prove that the group action under $\PGL(\vecr)$ is \emph{free} and \emph{proper}, and thus the resulting quotient set $\tensM_\vecr^*$ is a quotient manifold. {{\color{black}We also prove that the image $\tau(\tensM_\vecr^*)$ through TR mapping $\tau$ is an immersed submanifold of tensor space $\mathbb{R}^{n_1\times n_2\times\cdots\times n_d}$.}} \cref{tab:TT_TR_quotient} highlights the distinctions between the quotient geometry of the TT and TR decompositions.

\begin{table}[htbp]
    \centering
    \caption{Quotient geometries of TT and TR}\footnotesize
    \renewcommand{\arraystretch}{1.3} 
    \setlength{\tabcolsep}{5pt}       
    \begin{tabular}{lcc}
        \toprule
        {\sc Geometry} & {\sc Tensor train} & {\sc Tensor ring} \\
        \midrule
        \multirow{2}*{Parameter space} & \multicolumn{2}{c}{$\bartensM_\vecr=\mathbb{R}^{r_1\times n_1\times r_2}\times\mathbb{R}^{r_2\times n_2\times r_3}\times\cdots\times\mathbb{R}^{r_{d}\times n_d\times r_{d+1}}$} \\
        \cmidrule(rl){2-2}\cmidrule(rl){3-3}
         & $r_1=r_{d+1}=1$ & $r_1=r_{d+1}$\\
        \midrule
        \multirow{1}*{Total space} & \multicolumn{2}{c}{$\bartensM_\vecr^*=\mathbb{R}^{r_1\times n_1\times r_2}_*\times\mathbb{R}^{r_2\times n_2\times r_3}_*\times\cdots\times\mathbb{R}^{r_{d}\times n_d\times r_{1}}_*$} \\
        \cmidrule(rl){2-2}\cmidrule(rl){3-3}
        $(\tensU_1,\tensU_2,\dots,\tensU_d)$ & $\rank((\tensU_k)_{(3)}) = \rank((\tensU_{k+1})_{(1)}) = r_{k+1}$ & $\rank((\tensU_k)_{(2)})=r_{k}r_{k+1}$\\
        \midrule
        \multirow{1}*{Quotient manifold} & \multicolumn{2}{c}{$\tensM_\vecr^*=\bartensM_\vecr^*/\tensG$} \\
        \cmidrule(rl){2-2}\cmidrule(rl){3-3}
        Lie group $\tensG$ & $\GL(r_2)\times\GL(r_3)\times\cdots\times\GL(r_{d})$ & $(\GL(r_1)\times\cdots\times\GL(r_d))/\mathbb{R}_*$\\
        \cmidrule(rl){2-2}\cmidrule(rl){3-3}
        Dimension & $\sum_{k=1}^d r_{k}n_kr_{k+1}-\sum_{k=2}^{d} r_k^2$ & $\sum_{k=1}^d (r_{k}n_kr_{k+1}-r_k^2)+1$\\
        \midrule
        Reference & \cite[Theorem 14]{haegeman2014geometry} & \cref{thm:TRquotient}\\
        \bottomrule
    \end{tabular}
    \label{tab:TT_TR_quotient}
\end{table}

Moreover, we derive explicit parameterizations of the vertical and horizontal spaces, with the corresponding projection operators. Due to the ring structure of tensor ring decomposition, these constructions are different from those in tensor train. For instance, an element in the horizontal space of the quotient manifold in TT format is characterized by conditions on individual core tensors~\cite[eq.\,(26)]{uschmajew2013geometry}. However, the characterization of an element in the horizontal space of TR format involves coupling between two adjacent core tensors, which complicates the computation of the projections; see~\cref{prop:vertical_and_horizontal}. In addition, we extend all geometric results to the uniform tensor ring decomposition, where the number of parameters is independent of the order $d$. Despite the quotient geometry of uniform TR being known in~\cite[Theorem 20]{haegeman2014geometry}, we develop the parameterizations of the vertical and horizontal spaces.

The developed geometries facilitate optimization methods on the quotient manifold, and we numerically validate the quotient geometries through low-rank tensor ring completion. Numerical experiments report that the uniform TR-based methods require substantially fewer samples for successful recovery than TR-based methods, {{\color{black}and the quotient-based methods perform better than total space-based methods.}}

\paragraph{Organization}
\Cref{sec: preliminaries} provides basics for tensor computation and Riemannian geometry. We establish the quotient geometry of both TR and uniform TR decompositions in~\cref{sec:quotient}, and propose geometric methods in~\cref{sec:geometric_methods}. The developed geometries are numerically validated in~\cref{sec: numerical}. Finally, we draw the conclusion in~\cref{sec:conclusion}.

\section{Preliminaries}\label{sec: preliminaries}
In this section, we introduce the preliminaries of Riemannian geometry and basics for tensor computation. 

\subsection{Notation for Riemannian geometry}
Assume that a smooth manifold $\bartensM$ is embedded in a Euclidean space $\tensE$. The tangent space of $\bartensM$ at $x\in\bartensM$ is denoted by $\tangent_x\!\bartensM$. Let $\bartensM$ be endowed with a \emph{Riemannian metric} $\barg$, where $\barg_x:\tangent_x\!\bartensM\times\tangent_x\!\bartensM\to\mathbb{R}$ is a symmetric, bilinear, positive-definite function, and smooth with respect to $x\in\bartensM$. The Riemannian metric $g$ introduces a norm $\|\eta\|_x=\sqrt{\barg_x(\eta,\eta)}$ for $\eta\in\tangent_x\!\bartensM$. Given $\bar{\eta}\in\tangent_x\!\tensE\simeq\tensE$, the orthogonal projection operator onto $\tangent_x\!\bartensM$ is $\proj_{\tangent_x\!\bartensM}(\bar{\eta})$. The \emph{tangent bundle} is denoted by $\tangent\!\bartensM=\bigcup_{x\in\bartensM}\tangent_x\!\bartensM$. A smooth mapping $\bar{\retr}:\tangent\!\bartensM\to\bartensM$ is called a retraction~\cite[Definition 1]{absil2012projection} on $\bartensM$ around $x\in\bartensM$ if there exists of a neighborhood $U$ of $(x,0)\in\tangent\!\bartensM$ such that 1) $U\subseteq\mathrm{dom}(\bar{\retr})$ and $\bar{\retr}$ is smooth on $U$; 2) $\bar{\retr}_x(0) = x$ for all $x\in\bartensM$; 3) $\mathrm{D}\bar{\retr}_x(\cdot)[0]=\mathrm{id}_{\tangent_x\!\bartensM}$. The \emph{vector transport} operator is denoted by $\tensT_{y\gets x}:\tangent_x\!\bartensM\to\tangent_y\!\bartensM$.

We introduce the notation for the quotient manifold through \emph{Lie group actions}~\cite[\S 9]{boumal2023intromanifolds}, which is ubiquitous in tensor-related manifolds. The set of $r$-by-$r$ invertible matrices is denoted by $\GL(r)$, which is a special case of Lie group. A group action introduces an equivalence relation, defined as follows. Given a Riemannian manifold $(\bartensM,\barg)$, the orbit of $x\in\bartensM$ through the action $\theta:\tensG\times\bartensM\to\bartensM$ of $\tensG$ is the set $\tensF_x=\left\{y:\ y=\theta_h(x),\ h\in\tensG\right\}$. The group action induces a \emph{natural projection} $\pi(x)=[x]$ and an {{\color{black}equivalence relation}} $x\sim y\iff y=\theta_h(x)$ for some $h\in\tensG$ on $\bartensM$. Subsequently, we yield a quotient set $\tensM=\bartensM/\tensG$, which is a quotient manifold if the group action $\theta$ is \emph{free} and \emph{proper}. 

The vertical space at $x\in\bartensM$ is defined by $\vertical_x\!\bartensM:=\tangent_x\!\tensF_x=\mathrm{ker} (\mathrm{D}\pi(x))$, which is a subspace of $\tangent_x\!\bartensM$. The horizontal space is defined by the orthogonal complement of $\vertical_x\!\bartensM$, i.e., $\horizontal_x\!\bartensM:=(\vertical_x\!\bartensM)^\perp=\left\{\xi\in \tangent_x\!\bartensM:\ \barg_x(\xi,\eta)=0,\ \forall \eta\in\vertical_x\!\bartensM\right\}$. Then, it holds that $\tangent_x\!\bartensM=\vertical_x\!\bartensM\oplus\horizontal_x\!\bartensM$. Given a tangent vector $\eta\in\tangent_{[x]}\!\tensM$, we call the unique vector $\bareta\in\horizontal_x\!\bartensM$ satisfying $\mathrm{D}\pi(x)[\bareta]=\eta$ horizontal lift, denoted by $\bareta=\lift_x(\eta)$. For tangent vectors $\xi,\eta\in\tangent_{[x]}\!\tensM$, the Riemannian metric of $\tensM$ is defined by $g_{[x]}(\xi,\eta):=\barg_x(\lift_x(\xi),\lift_x(\eta))$. A retraction on $\tensM$ is given by $\retr_{[x]}(\eta):=[\bar{\retr}_x(\lift_x(\eta))]$, where $\bar{\retr}$ is a retraction on $\bartensM$.

\subsection{Notation for tensor operations}
The following notations are involved in tensor computations~\cite{kolda2009tensor}. The mode-$k$ unfolding of a tensor $\mathcal{X} \in \mathbb{R}^{n_1 \times n_2\times\cdots\times n_d}$ is denoted by $ \matx_{(k)}\in\mathbb{R}^{n_k\times n_{-k}} $, where $n_{-k}:=\prod_{i\neq k}n_i$. The $ (i_1,i_2,\dots,i_d)$-th entry of $\mathcal{X}$ corresponds to the $(i_k,j)$-th entry of $ \matx_{(k)} $, where
$ j = \pi_k(i_1,\dots,i_{k-1},i_{k+1},\dots,i_d) :=1 + \sum_{\ell \neq k, \ell = 1}^d(i_\ell-1)\prod_{m = 1, m \neq k}^{\ell-1} n_m$ for $k\in[d]:=\{1,2,\dots,d\}$. 
The inner product of $\tensX,\tensY\in\mathbb{R}^{n_1\times\cdots\times n_d}$ is defined by $\langle\tensX,\tensY\rangle := \sum_{i_1=1}^{n_1}\cdots\sum_{i_d=1}^{n_d}\tensX({i_1,\dots,i_d})\tensY({i_1,\dots,i_d})$. 
The Frobenius norm of a tensor $\tensX$ is defined by $\|\tensX\|_\mathrm{F}:=\sqrt{\langle\tensX,\tensX\rangle}$.
Given a third-order tensor $\tensU\in\mathbb{R}^{n_1\times n_2\times n_3}$ and two matrices $\mata\in\mathbb{R}^{r_1\times n_1}$ and $\matb\in\mathbb{R}^{r_3\times n_3}$, it holds that 
\begin{equation*}
	\begin{aligned}
		(\tensU\times_1\mata\times_3\matb)_{(1)}&=\mata\matu_{(1)}(\matb\otimes\matI_{n_2})^\top,\\
		(\tensU\times_1\mata\times_3\matb)_{(2)}&=\matu_{(2)}(\matb\otimes\mata)^\top,\\
		(\tensU\times_1\mata\times_3\matb)_{(3)}&=\matb\matu_{(3)}(\matI_{n_2}\otimes\mata)^\top.
	\end{aligned}
\end{equation*} 
{{\color{black}For an indexed tensor $\tensU_k$, the mode-$\ell$ unfolding is denoted by $(\tensU_k)_{(\ell)}$. For a tensor $\tensU$ without a subscript, the mode-$\ell$ unfolding is denoted by $\matu_{(\ell)}$.}}

\subsection{Tensor ring decomposition}
Given a tensor $\tensX\in\mathbb{R}^{n_1\times n_2\times\cdots\times n_d}$, recall that tensor ring decomposition~\cite{zhao2016tensor} decomposes $\tensX$ into $d$ core tensors $\tensU_k\in\mathbb{R}^{r_k\times n_k\times r_{k+1}}$ in~\cref{eq: elementwise TR}, i.e., $\tensX=\llbracket\tensU_1,\tensU_2,\dots,\tensU_d\rrbracket$. The parameter space of tensor ring decomposition is a product space of third-order tensors
\[\bartensM_\vecr=\mathbb{R}^{r_1\times n_1\times r_2}\times \mathbb{R}^{r_2\times n_2\times r_3}\times\cdots\times\mathbb{R}^{r_d\times n_d\times r_1}.\]
The tensor ring rank is a $d$-dimensional array of integers defined by
{{\color{black}
\begin{equation*}
	\begin{aligned}
		\ranktr(\tensX):=\{\vecr=(r_1,r_2,\dots,r_d):\ &\tensX=\llbracket\tensV_1,\tensV_2,\dots,\tensV_d\rrbracket\ \text{with}\ \tensV_k\in\mathbb{R}^{\widehat{r}_k\times n_k\times\widehat{r}_{k+1}},\\
		&\text{if }\widehat r_k\leq r_k, \text{ then } \widehat{r}_k=r_k\ \text{for}\ k=1,2,\dots,d\},
	\end{aligned}
\end{equation*}
i.e., an array of the smallest integers such that $\tensX$ admits a TR decomposition with core tensors in $\bartensM_\vecr$.}} Note that the TR rank of $\tensX$ can be nonunique~\cite[\S 9]{ye2018tensor} due to the ring structure.

For the sake of brevity, we introduce the following matrices for the core tensors $\{\tensU_k\}$: 1) $\matW_k:=(\tensU_k)_{(2)}\in\mathbb{R}^{n_k\times (r_kr_{k+1})}$ for mode-2 unfolding matrices of one core tensor; 2) $\matW_{\neq k}:=(\tensU_{\neq k})_{(2)}$ with $\tensU_{\neq k}\in\mathbb{R}^{r_{k-1}\times\prod_{j\neq k}n_j\times r_{k}}$ defined by the {{\color{black}slice}} matrices $\matU_{\neq k}\left(\pi_k(i_1,\dots,i_{k-1},i_{k+1},\dots,i_d)\right){:=}\left(\prod_{j=k+1}^{d}\matU_{j}(i_{j})\prod_{j=1}^{k-1}\matU_{j}(i_{j})\right)^\top$ for $i_j\in[n_j]:=\{1,2,\dots,n_j\}$ and $j\in[d]$. In the light of the cyclic symmetry of trace operator, we have
\begin{equation}\label{eq: matricization of X}
	\matx_{(k)}=\matW_k^{}\matW_{\neq k}^\top.
\end{equation}

\section{Quotient geometry of tensor ring decomposition}\label{sec:quotient}
In this section, we establish the quotient geometry of tensor ring decomposition by the following steps: 1) inspired by the gauge invariance of the TR decomposition and the fundamental theorem of MPS, we define a parameter space by imposing full-rank conditions; 2) we introduce an equivalence on the new parameter space, and we prove that the resulting quotient set is a quotient manifold; 3) we provide parametrizations of the vertical and horizontal spaces and projections; {{\color{black}4) we show that the image of the quotient manifold under TR mapping is an immersed submanifold of tensor space; 5) we extend the results to the uniform TR decomposition. Additionally, we discuss the connections to the tensor train decomposition.}}

\subsection{Gauge invariance}
Given a TR tensor $\tensX=\llbracket\tensU_1,\dots,\tensU_d\rrbracket$, we observe that the TR decomposition of $\tensX$ can be non-unique since 
\begin{equation*}
    \tr(\matu_1(i_1)\matu_2(i_2)\cdots\matu_d(i_d))=\tr(\mata_1^{}\matu_1(i_1)\mata_2^{-1}\mata_2^{}\matu_2(i_2)\mata_3^{-1}\cdots\mata_d^{}\matu_d(i_d)\mata_1^{-1})
\end{equation*}
holds for all invertible matrices $\mata_k\in\GL(r_k)$ with $k\in[d]$. In other words, inserting an invertible matrix and its inverse between two adjacent core tensors {{\color{black}leaves}} the full tensor unchanged, which is referred to as the \emph{gauge invariance}~\cite{chen2020tensor}. \cref{fig:gauge_inv} provides an illustration of the gauge invariance via the tensor networks.

\begin{figure}[htbp]
    \centering
    \begin{tikzpicture}\scriptsize
        \coordinate (O) at (0,0); 
        \def\l{1.6};


		\node[circle,draw,inner sep=1pt] (U1) at (O) {$\tensU_1$};
		\node[circle,draw,inner sep=1pt] (U2) at ($(O)+(0.5*\l,-0.866*\l)$) {$\tensU_2$};
		\node[circle,draw,inner sep=1pt] (U3) at ($(O)+(1.5*\l,-0.866*\l)$) {$\tensU_3$};
		\node[circle,draw,inner sep=1pt] (U4) at ($(O)+(2*\l,0)$) {$\tensU_4$};
		\node[circle,draw,inner sep=1pt] (U5) at ($(O)+(1.5*\l,0.866*\l)$) {$\tensU_5$};
		\node[circle,draw,inner sep=1pt] (U6) at ($(O)+(0.5*\l,0.866*\l)$) {$\tensU_{d}$};
        
        \draw[-] (U6)--(U1)--(U2)--(U3)--(U4)-- (U5);
        \draw[dashed] (U5)--(U6);
        \draw (U1) -- ($(U1)+0.5*(-\l,0)$);
        \draw (U2) -- ($(U2)+0.5*(-0.5*\l,-0.866*\l)$);
        \draw (U3) -- ($(U3)+0.5*(0.5*\l,-0.866*\l)$);
        \draw (U4) -- ($(U4)+0.5*(\l,0)$);
        \draw (U5) -- ($(U5)+0.5*(0.5*\l,0.866*\l)$);
        \draw (U6) -- ($(U6)+0.5*(-0.5*\l,0.866*\l)$);

        \node at ($(3*\l,0)$) {\Large$=$};

        \coordinate (O) at ($(4*\l,0)$);
        \node[circle,draw,inner sep=1pt] (U1) at (O) {$\tensU_1$};
		\node[circle,draw,inner sep=1pt] (U2) at ($(O)+(0.5*\l,-0.866*\l)$) {$\tensU_2$};
		\node[circle,draw,inner sep=1pt] (U3) at ($(O)+(1.5*\l,-0.866*\l)$) {$\tensU_3$};
		\node[circle,draw,inner sep=1pt] (U4) at ($(O)+(2*\l,0)$) {$\tensU_4$};
		\node[circle,draw,inner sep=1pt] (U5) at ($(O)+(1.5*\l,0.866*\l)$) {$\tensU_5$};
		\node[circle,draw,inner sep=1pt] (U6) at ($(O)+(0.5*\l,0.866*\l)$) {$\tensU_{d}$};
        
        \coordinate (A1) at ($0.35*(U6)+0.65*(U1)$);
        \coordinate (A2) at ($0.35*(U1)+0.65*(U2)$);
        \coordinate (A3) at ($0.35*(U2)+0.65*(U3)$);
        \coordinate (A4) at ($0.35*(U3)+0.65*(U4)$);
        \coordinate (A5) at ($0.35*(U4)+0.65*(U5)$);
        \coordinate (A6) at ($0.35*(U5)+0.65*(U6)$);
        
        \coordinate (A1inv) at ($0.35*(U1)+0.65*(U6)$);
        \coordinate (A2inv) at ($0.35*(U2)+0.65*(U1)$);
        \coordinate (A3inv) at ($0.35*(U3)+0.65*(U2)$);
        \coordinate (A4inv) at ($0.35*(U4)+0.65*(U3)$);
        \coordinate (A5inv) at ($0.35*(U5)+0.65*(U4)$);
        \coordinate (A6inv) at ($0.35*(U6)+0.65*(U5)$);
        
        \fill (A1) circle (2pt);
        \fill (A2) circle (2pt);
        \fill (A3) circle (2pt);
        \fill (A4) circle (2pt);
        \fill (A5) circle (2pt);
        \fill (A6) circle (2pt);

        \fill (A1inv) circle (2pt);
        \fill (A2inv) circle (2pt);
        \fill (A3inv) circle (2pt);
        \fill (A4inv) circle (2pt);
        \fill (A5inv) circle (2pt);
        \fill (A6inv) circle (2pt);
        
        
        \draw (U1) -- ($(U1)+0.5*(-\l,0)$);
        \draw (U2) -- ($(U2)+0.5*(-0.5*\l,-0.866*\l)$);
        \draw (U3) -- ($(U3)+0.5*(0.5*\l,-0.866*\l)$);
        \draw (U4) -- ($(U4)+0.5*(\l,0)$);
        \draw (U5) -- ($(U5)+0.5*(0.5*\l,0.866*\l)$);
        \draw (U6) -- ($(U6)+0.5*(-0.5*\l,0.866*\l)$);

        \draw[-] (A6)--(U6)--(U1)--(U2)--(U3)--(U4)-- (U5)--(A6inv);
        \draw[dashed] (A6inv)--(A6);

        \node at ($(A1)+0.2*(-0.866*\l,0.5*\l)$) {$\mata_1$};
        \node at ($(A1inv)+0.2*(-0.866*\l,0.5*\l)$) {$\mata_1^{-1}$};
        
        \node at ($(A2)+0.2*(-0.866*\l,-0.5*\l)$) {$\mata_2$};
        \node at ($(A2inv)+0.2*(-0.866*\l,-0.5*\l)$) {$\mata_2^{-1}$};
        
        \node[below] at ($(A3)$) {$\mata_3$};
        \node[below] at ($(A3inv)$) {$\mata^{-1}_3$};
        
        \node at ($(A4)+0.2*(0.866*\l,-0.5*\l)$) {$\mata_4$};
        \node at ($(A4inv)+0.2*(0.866*\l,-0.5*\l)$) {$\mata_4^{-1}$};
        
        \node at ($(A5)+0.2*(0.866*\l,0.5*\l)$) {$\mata_5$};
        \node at ($(A5inv)+0.2*(0.866*\l,0.5*\l)$) {$\mata_5^{-1}$};
        
        \node[above] at ($(A6)$) {$\mata_{d}$};
        \node[above] at ($(A6inv)$) {$\mata^{-1}_6$};
        
    \end{tikzpicture}
    \caption{Gauge invariance of tensor ring decomposition.}
    \label{fig:gauge_inv}
\end{figure}

An instinctive question is whether the TR decomposition of $\tensX$ is unique up to the gauge invariance. In fact, the uniqueness can be guaranteed by the \emph{injectivity} of $\tensX$; see~\cref{thm: fundamental}. The result is referred to as the \emph{fundamental theorem of MPS} in~\cite[Theorem 1]{molnar2018normal} {{\color{black}in quantum physics. 

Specifically, the particles, physical dimensions, and bond dimensions in~\cite{molnar2018normal} correspond to the tensor modes, mode sizes $n_k$, and TR ranks $r_1,r_2,\dots,r_d$, respectively. The matrices associated with the physical index are the slice matrices $\matu_k(i_k)=\tensU_k(:,i_k,:)$. The injectivity condition in~\cite{molnar2018normal} states that the linear mapping from the virtual indices to the physical index is injective. The matrix representation is $(\tensU_k)_{(2)}$ in our notation, and the injectivity is equivalent to the full-rank condition, i.e., the core tensors $\tensU_k$ satisfy $\rank((\tensU_k)_{(2)})=r_kr_{k+1}$ and $r_kr_{k+1}\leq n_k$ in~\cref{thm: fundamental}. 

The fundamental theorem of MPS~\cite[Theorem~1]{molnar2018normal} states that two injective MPS represent the same tensor if and only if they are related by a gauge transformation. In our notation, the theorem translates into the statement that two injective TR tensors $\llbracket\tensU_1,\tensU_2,\dots,\tensU_d\rrbracket$ and $\llbracket\tensV_1,\tensV_2,\dots,\tensV_d\rrbracket$ represent the same tensor if and only if there exist invertible matrices $\matA_k$ such that $\tensV_k=\tensU_k\times_1\matA_k\times_3\matA_{k+1}^{-\top}$ for all $k\in[d]$.}}

\begin{theorem}[fundamental theorem of MPS]\label{thm: fundamental}
	Let $\tensX=\llbracket\tensU_1,\tensU_{2},\dots,\tensU_d\rrbracket$ satisfy $\rank((\tensU_k)_{(2)})=r_kr_{k+1}$ and $r_kr_{k+1}\leq n_k$ for all $k\in[d]$ and $d\geq 3$. If $\tensX$ can also be expressed by $\llbracket\tensV_1,\tensV_2,\dots,\tensV_d\rrbracket$ with $\tensV_k\in\mathbb{R}^{r_k\times n_k\times r_{k+1}}$ and $\rank((\tensV_k)_{(2)})=r_kr_{k+1}$, there exists $\mata_k\in\GL(r_k)$ with $k\in[d]$, such that 
	\[\matv_k(i_k)=\mata_k^{}\matu_k(i_k)\mata_{k+1}^{-1}\text{, or equivalently, }\tensV_k^{}=\tensU_k^{}\times_1^{}\mata_k^{}\times_3\mata_{k+1}^{-\top},\]
    where $\mata_{d+1}:=\mata_1$. Moreover, $\vecmatA:=(\mata_1,\mata_2,\dots,\mata_d)$ is unique up to a constant.
\end{theorem}

Note that the result no longer holds for the matrix case $d=2$. For instance, {{\color{black}given an invertible matrix $\mata\in\GL(r_1r_2)$ that can not be represented by Kronecker product of two matrices in $\GL(r_1)$ and $\GL(r_2)$, we}} consider $\matx=\llbracket\tensU_1,\tensU_2\rrbracket$ and $\tensV_1,\tensV_2$ and construct $\tensV_1,\tensV_2$ by $(\tensV_1)_{(2)}=(\tensU_1)_{(2)}\mata$ and $(\tensV_2)_{(2)}=(\tensU_2)_{(2)}\mata^{-\top}$. It holds that $\llbracket\tensU_1,\tensU_2\rrbracket=\llbracket\tensV_1,\tensV_2\rrbracket$, but~\cref{thm: fundamental} does not hold.

In the light of~\cref{thm: fundamental}, we impose the full-rank condition on $\bartensM_\vecr$ and define a new parameter space 
\[\bartensM_\vecr^*=\mathbb{R}_*^{r_1\times n_1\times r_2}\times \mathbb{R}_*^{r_2\times n_2\times r_3}\times\cdots\mathbb{R}_*^{r_d\times n_d\times r_1},\]
where $\mathbb{R}_*^{r_k\times n_k\times r_{k+1}}=\{\tensU_k\in\mathbb{R}^{r_k\times n_k\times r_{k+1}}:\rank((\tensU_k)_{(2)})=r_kr_{k+1}\}$ and $r_kr_{k+1}\leq n_k$. 
\begin{definition}[injective TR tensor]
    A tensor $\tensX=\llbracket\tensU_1,\tensU_2,\dots,\tensU_d\rrbracket$ with $\vectensU=(\tensU_1,\tensU_2,\dots,\tensU_d)\in\bartensM_\vecr^*$ is referred to as an injective TR tensor.
\end{definition}

\begin{lemma}\label{lem: rank}
	Given an injective TR tensor $\tensX=\llbracket\tensU_1,\tensU_2,\dots,\tensU_d\rrbracket$ with $\vectensU\in\bartensM_\vecr^*$, it holds that $\rank(\matx_{(k)})=\rank(\matW_k)=\rank(\matW_{\neq k})=r_kr_{k+1}$, {{\color{black}$\rank((\tensU_k)_{(1)})=r_k$ and $\rank((\tensU_k)_{(3)})=r_{k+1}$}} for $k\in[d]$.
\end{lemma}
\begin{proof}
	It suffices to prove $\rank(\matW_{\neq k})=r_kr_{k+1}$ for $k\in[d]$. Specifically, we prove the result for $k=1$ and others in a same fashion. Recall the rows of $\matW_{\neq k}$. We aim to prove $\mathbb{R}^{r_2\times r_1}=\Span\{\prod_{j=2}^{d}\matu_j(i_j)\}_{i_2,\dots,i_d=1}^{n_2,\dots,n_d}$, i.e., any matrix in $\mathbb{R}^{r_2\times r_1}$ can be represented by the linear combinations of matrices $\prod_{j=2}^{d}\matu_j(i_j)$. Or equivalently, $\matv\in\Span^\perp\{\prod_{j=2}^{d}\matu_j(i_j)\}_{i_2,\dots,i_d=1}^{n_2,\dots,n_d}$ if and only if $\matv=0$ by induction.
	
	We observe that
	\[0=\langle\matv,\prod_{j=2}^{d}\matu_j(i_j)\rangle=\langle\matv(\prod_{j=3}^{d}\matu_j(i_j))^\top,\matu_2(i_2)\rangle\]
	holds for all $i_2\in[n_2],\dots,i_d\in[n_d]$. Since $\rank(\matW_2)=r_2r_3$, we have $\mathbb{R}^{r_2\times r_3}=\Span\{\matu_2(1),\dots,\matu_2(n_2)\}$. Therefore, it holds that $\matv(\prod_{j=3}^{d}\matu_j(i_j))^\top=0$ for all $i_3,i_4,\dots,i_d$. 
	
	Suppose $\matv(\prod_{j=l}^{d}\matu_j(i_j))^\top=0$ for some integer $l\geq 2$ and $l<d$, we aim to prove that $\matv(\prod_{j=l+1}^{d}\matu_j(i_j))^\top=0$.
	Since $\mathbb{R}^{r_l\times r_{l+1}}=\Span\{\matu_l(1),\dots,\matu_l(n_l)\}$, we have $\matv(\prod_{j=l+1}^{d}\matu_j(i_j))^\top\tilde{\matu}_l^\top=0$ holds for all $\tilde{\matu}_l\in\mathbb{R}^{r_l\times r_{l+1}}$. Therefore, it holds that $\matv(\prod_{j=l+1}^{d}\matu_j(i_j))^\top=0$ and we obtain $\matv=0$ by induction. Consequently, we have 
	\[r_kr_{k+1}=\rank(\matW_{\neq k})=\rank(\matW_k^\dagger\matx_{(k)}^{})\leq\rank(\matx_{(k)})\leq\rank(\matW_k)=r_kr_{k+1}\]
	and thus $\rank(\matx_{(k)})=r_kr_{k+1}$, where $\matW_k^\dagger=(\matW_k^\top\matW_k^{})^{-1}\matW_k^{\top}$ is the Moore--Penrose pseudoinverse of $\matW_k$. 

	{{\color{black}We aim to prove that $\rank((\tensU_k)_{(1)})=r_k$ and $\rank((\tensU_k)_{(3)})=r_{k+1}$ from the full-rank condition $\rank((\tensU_k)_{(2)})=r_kr_{k+1}$ for all $k\in[d]$. 

	Specifically, since the full-rank condition $\rank((\tensU_k)_{(2)})=r_kr_{k+1}$ holds, it holds that $\Span\{\matu_{k}(i_k)\}_{i_k=1}^{n_k}=\mathbb{R}^{r_k\times r_{k+1}}$. For mode-1 unfolding, it holds that
	\[(\tensU_k)_{(1)}=[\matu_k(1)\ \matu_k(2)\ \dots\ \matu_k(n_k)]\Pi_k,\]
	where $\Pi_k\in\mathbb{R}^{(n_kr_{k+1})\times (n_kr_{k+1})}$ is a permutation matrix. If $\vecx^\top(\tensU_k)_{(1)}=0$ for some $\vecx\in\mathbb{R}^{r_k}$, it holds that $\vecx^\top\matu_k(i_k)=0$ for all $i_k\in[n_k]$ and thus $\vecx^\top\mata=0$ for all $\mata\in\mathbb{R}^{r_k\times r_{k+1}}$. Therefore, we obtain that $\vecx=0$ and thus $\rank((\tensU_k)_{(1)})=r_k$. 

	Similarly, for mode-3 unfolding, it holds that
	\[(\tensU_k)_{(3)}=[\matu_k(1)^\top\ \matu_k(2)^\top\ \dots\ \matu_k(n_k)^\top].\]
	Since $\Span\{\matu_{k}(i_k)^\top\}_{i_k=1}^{n_k}=\mathbb{R}^{r_{k+1}\times r_k}$, we can show that $\rank((\tensU_k)_{(3)})=r_{k+1}$ in a same fashion.}}
\end{proof}

By using~\cref{lem: rank}, we can prove that $\vecr\in\ranktr(\tensX)$. {{\color{black}In other words, if $\tensX$ is an injective TR tensor, $\vecr=(r_1,r_2,\dots,r_d)$ is a TR rank of $\tensX$.}}
\begin{proposition}
	Given a $d$-th order injective TR tensor $\tensX=\llbracket\tensU_1,\tensU_2,\dots,\tensU_d\rrbracket$ with $\vectensU\in\bartensM_\vecr^*$, it holds that $\vecr\in\ranktr(\tensX)$. 
\end{proposition}
\begin{proof}
	{{\color{black}If $\tensX$ can be decomposed by $\tensX=\llbracket\tensV_1,\tensV_2,\dots,\tensV_d\rrbracket$ with $\tensV_k\in\mathbb{R}^{\widehat{r}_k\times n_k\times\widehat{r}_{k+1}}$ and $\widehat{\vecr}\leq\vecr$, it follows from Lemma~\ref{lem: rank} that $r_kr_{k+1}=\rank(\matx_{(k)})\leq\widehat{r}_{k}\widehat{r}_{k+1}\leq r_kr_{k+1}$ for $k\in[d]$. Therefore, $r_k=\widehat{r}_{k}$.}}
\end{proof}


\subsection{Quotient geometry}
Recall that the manifold structure of several special cases of TR decomposition was studied. 
However, due to the underlying ring structure, the existing result cannot be straightforwardly applied to TR decomposition. Therefore, we aim to develop the manifold structure for TR decomposition by developing quotient manifold structure on the total space $\bartensM_\vecr^*$.

First, since the TR decomposition is non-unique, an appropriate equivalence class is necessary to eliminate such uncertainty. In the light of the gauge invariance and~\cref{thm: fundamental}, we consider the set $\bartensM_\vecr^*$ to be the total space, and the set
\[\PGL(\vecr):=\GL(\vecr)/\mathbb{R}_*\quad \text{with} \quad \GL(\vecr)=\GL(r_1)\times\GL(r_2)\times\cdots\times\GL(r_d)\]
to introduce an {{\color{black}equivalence relation}} on $\bartensM_\vecr^*$, where $\mathbb{R}_*=\mathbb{R}\setminus\{0\}$. A representative in $\PGL(\vecr)$ is defined by $[\vecmatA]:=\{c\vecmatA:c\in\mathbb{R}_*,\vecmatA\in\GL(\vecr)\}$. $\PGL(\vecr)$ is a quotient manifold and a Lie group. 

Subsequently, we consider the group action on $\bartensM_\vecr^*$:
\[\theta_{[\vecmatA]}(\vectensU):=\theta_{\vecmatA}(\vectensU)={{\color{black}(\theta^1_{\vecmatA}(\tensU_1),\theta^2_{\vecmatA}(\tensU_2),\dots,\theta^d_{\vecmatA}(\tensU_d))}}\quad \text{for} \ \vectensU\in\bartensM_\vecr^*,\]
where  $\theta^k_{\vecmatA}(\tensU_k)=\tensU_k\times_1\matA_k\times_3\matA_{k+1}^{-\top}$ for $k\in[d]$ and $\mata_{d+1}=\mata_1$. The group action introduces an {{\color{black}equivalence relation}}: $\vectensU\sim\vectensV\iff\vectensV=\theta_{[\vecmatA]}(\vectensU)$ for some $[\vecmatA]\in\PGL(\vecr)$. Note that the group action is well-defined since it does not depend on a specific representative, i.e., $\theta^k_{\vecmatA}(\tensU_k)=\theta^k_{c\vecmatA}(\tensU_k)$ for all $c\in\mathbb{R}_*$. 

\begin{remark}[Necessity of the projective group]
    It is worth noting that the necessity of the projective group $\PGL(\vecr)$ is specific to the tensor ring decomposition. For TT decomposition, the group action induced from $\GL(\vecr)$ is free and proper. However, the ring structure of TR introduces a redundancy that relates to all cores simultaneously. Specifically, it holds that $\theta_{\vecmatA}(\vectensU)=\theta_{c\vecmatA}(\vectensU)$ for all $c\neq 0$. The redundancy renders the group action from $\GL(\vecr)$ a non-free group action for TR. Therefore, it is necessary to consider the projective group $\PGL(\vecr)$ to ensure that the resulting group action $\theta$ is free and proper.
\end{remark}

Then, we aim to prove that $\tensM_\vecr^*=\bartensM_\vecr^*/\PGL(\vecr)$ is a smooth quotient manifold. It suffices to verify that $\theta$ is free and proper~\cite[\S 9.2]{boumal2023intromanifolds}. 
\begin{proposition}\label{prop: free}
	$\theta$ is a free group action, i.e., for all $\vectensU\in\bartensM_\vecr^*$, $\theta_{[\vecmatA]}(\vectensU)=\vectensU\Rightarrow\vecmatA=(c\matI_{r_1},\dots,c\matI_{r_d})$ for some $c\in\mathbb{R}_*$.
\end{proposition}
\begin{proof}
	By~\cref{thm: fundamental}, there exists $\vecmatA\in\GL(\vecr)$, such that $\tensU_k\times_1\mata_k\times_3\mata_{k+1}^{-\top}=\tensU_k$ for all $k\in[d]$, i.e., $(\tensU_k)_{(2)}(\mata_{k+1}^{-1}\otimes\mata_{k}^{\top})=(\tensU_k)_{(2)}$. Since $(\tensU_k)_{(2)}$ is of full rank, we obtain $\mata_{k+1}^{-1}\otimes\mata_{k}^{\top}=\matI_{r_kr_{k+1}}$. Therefore, there exists $c_k\in\mathbb{R}_*$, such that $\mata_{k}=c_k\matI_{r_k}$ and $\mata_{k+1}=c_k\matI_{r_{k+1}}$. It follows from the arbitrariness of $k$ that $c_1=c_2=\cdots=c_d$. 
\end{proof}


\begin{proposition}\label{prop: proper}
    $\theta$ is a proper group action. Or equivalently, given a sequence $\{\vectensU^{(t)}\}_{t=1}^{\infty}$ in $\bartensM_\vecr^*$ converging to $\vectensU^*$, and a sequence $\{[\vecmatA^{(t)}]\}_{t=1}^{\infty}$ in the projected group $\PGL(\vecr)$. If the sequence $\vectensV^{(t)}:=\{\theta_{[\vecmatA^{(t)}]}(\vectensU^{(t)})\}_{t=1}^{\infty}$, generated from the group action $\theta:([\vecmatA^{(t)}],\vectensU^{(t)})\mapsto\theta_{[\vecmatA^{(t)}]}(\vectensU^{(t)})$, converges to $\vectensV^*\in\bartensM_\vecr^*$, then the sequence $\{[\vecmatA^{(t)}]\}_{t=1}^{\infty}$ converges to a point $[\vecmatA^*]\in\PGL(\vecr)$.
\end{proposition}
\begin{proof}	
    The proof proceeds in two steps: we first show that the sequence $\{[\vecmatA^{(t)}]\}$ is precompact in $\PGL(\vecr)$ through bounded singular values, and then prove convergence by identifying the unique accumulation point. 
    
	First, we prove that $\{[\vecmatA^{(t)}]\}$ has an accumulation point. We assume $\|\mata_1^{(t)}\|_\frob=1$ to eliminate the equivalence class in $\PGL(\vecr)$. The inequalities $\|\tensV_k^{(t)}-\tensV_k^*\|_\frob\leq\|\tensV_k^*\|_\frob/2$ holds and the smallest singular value satisfies $\sigma_{\min}((\tensU_k^{(t)})_{(2)})\geq\sigma_{\min}((\tensU_k^*)_{(2)})/2>0$ for some sufficiently large $t$. Therefore, it follows from $(\mata_{k+1}^{(t)})^{-1}\otimes(\mata_{k}^{(t)})^{\top}=(\tensU_k^{(t)})_{(2)}^\dagger(\tensV_k^{(t)})_{(2)}$ that
		\begin{align}
		\label{eq: ratio of singular values}
			\frac{\sigma_{\max}(\mata_{k}^{(t)})}{\sigma_{\min}(\mata_{k+1}^{(t)})}&=\|(\mata_{k+1}^{(t)})^{-1}\otimes(\mata_{k}^{(t)})^{\top}\|_2=\left\|(\tensU_k^{(t)})_{(2)}^\dagger(\tensV_k^{(t)})_{(2)}\right\|_2\nonumber\\
			&\leq\left\|(\tensU_k^{(t)})_{(2)}^\dagger\right\|_2\|\tensV_k^{(t)}\|_\frob\leq\left\|(\tensU_k^{(t)})_{(2)}^\dagger\right\|_2(\frac12\|\tensV_k^*\|_\frob+\|\tensV_k^*\|_\frob)\nonumber\\
			&\leq\frac{3\|\tensV_k^*\|_\frob}{\sigma_{\min}((\tensU_k^*)_{(2)})}
		\end{align}
	for $k\in[d]$, where $(\tensU_k^{(t)})_{(2)}^\dagger$ is the Moore--Penrose inverse of $(\tensU_k^{(t)})_{(2)}$. {{\color{black}Denote
    \[C:=\max_{k\in[d]}\frac{3\|\tensV_k^*\|_\frob}{\sigma_{\min}((\tensU_k^*)_{(2)})}\]
    and $\mata_{d+1}=\mata_1$.}} Therefore, we yield that $\{\sigma_{\max}(\mata_{1}^{(t)})\}_{t=1}^\infty$ is bounded above and 
    \begin{align*}
        \|\mata_{k}^{(t)}\|_\frob&\leq\sqrt{r_k}\sigma_{\max}(\mata_{k}^{(t)})\leq\sqrt{r_k}C\sigma_{\max}(\mata_{k+1}^{(t)})\leq\sqrt{r_k}C^2\sigma_{\max}(\mata_{k+2}^{(t)})\leq\cdots\\
        &\leq\sqrt{r_k}C^{d-k}\sigma_{\max}(\mata_{d}^{(t)})\leq\sqrt{r_k}C^{d+1-k}\sigma_{\max}(\mata_{1}^{(t)})\leq\sqrt{r_k}C^{d+1-k}\|\mata_1^{(t)}\|_\frob\\
		&\leq\sqrt{r_k}C^{d+1-k}
    \end{align*}
	holds for all $k=2,3,\dots,d$ from recursively applying~\cref{eq: ratio of singular values}. 
	{{\color{black}Since the singular value $\sigma_{\max}(\mata_{1}^{(t)})$ is bounded above by $1$ and $\|\mata_1^{(t)}\|_\frob=1$ by assumption, it follows from~\cref{eq: ratio of singular values} that $\sigma_{\max}(\mata_{k}^{(t)})\leq C^{d-k+1}\sigma_{\max}(\mata_{1}^{(t)})\leq C^{d-k+1}$. By multiplying~\cref{eq: ratio of singular values} for $k\in[d]$, we obtain that 
	\[\frac{\sigma_{\max}(\mata_k^{(t)})}{\sigma_{\min}(\mata_k^{(t)})}\leq\prod_{k=1}^d\frac{\sigma_{\max}(\mata_k^{(t)})}{\sigma_{\min}(\mata_k^{(t)})}=\prod_{k=1}^d\frac{\sigma_{\max}(\mata_k^{(t)})}{\sigma_{\min}(\mata_{k+1}^{(t)})}\leq C^d\]
	holds for $k\in[d]$ and sufficiently large $t$. Therefore, since $\sigma_{\max}(\mata_1^{(t)})$ is bounded below by $1/\sqrt{r_1}$, we obtain that 
	\[\sigma_{\min}(\mata_k^{(t)})\geq\frac{\sigma_{\max}(\mata_k^{(t)})}{C^d}\geq\frac{\sigma_{\max}(\mata_{1}^{(t)})}{C^{d+k-1}}\geq\frac{1}{C^{d+k-1}\sqrt{r_1}}.\]}}Consequently, it follows from the Weierstrass Theorem that $\{\vecmatA^{(t)}\}$ has an accumulation point $\vecmatA^*$ in $\GL(\vecr)$ and thus $\{[\vecmatA^{(t)}]\}$ has an accumulation point $[\vecmatA^*]$ {{\color{black}in $\PGL(\vecr)$}}. 
	
	Second, we aim to prove that $\{[\vecmatA^{(t)}]\}$ converges to $[\vecmatA^*]$ in $\PGL(\vecr)$ by constructing auxiliary sequences $\{\tilde{\mata}_k^{(t)}\}_{t=0}^{\infty}$ for $k\in[d]$. We observe that $(\mata_{k+1}^{(t)})^{-1}\otimes(\mata_{k}^{(t)})^{\top}=(\tensU_k^{(t)})_{(2)}^\dagger(\tensV_k^{(t)})_{(2)}$ converges to $(\tensU_k^*)_{(2)}^\dagger(\tensV_k^*)_{(2)}$. Therefore, it holds that $(\mata_{k+1}^*)^{-1}\otimes(\mata_{k}^*)^{\top}=(\tensU_k^*)_{(2)}^\dagger(\tensV_k^*)_{(2)}$, which is invertible. We obtain that 
    \[\lim_{t\to\infty}(\mata_{k+1}^*(\mata_{k+1}^{(t)})^{-1})\otimes(\mata_{k}^{(t)}(\mata_{k}^*)^{-1})^{\top}=\matI_{r_kr_{k+1}}=\matI_{r_{k+1}}\otimes\matI_{r_k}.\] 
    Denote $\tilde{\mata}_k^{(t)}:=\mata_{k}^{(t)}(\mata_{k}^*)^{-1}$, it holds that $(\tilde{\mata}_{k+1}^{(t)})^{-1}\otimes(\tilde{\mata}_k^{(t)})^\top$ converges to $\matI_{r_kr_{k+1}}$. 
	
	{{\color{black}Before moving forward, we prove the boundedness of the singular values of the auxiliary sequence. Based on the boundedness of the singular values of $\mata_k^{(t)}$, we can prove the boundedness of the singular values of $\tilde{\mata}_k^{(t)}$ as follows.
	\begin{equation*}
		\begin{aligned}
			\sigma_{\max}(\tilde{\mata}_k^{(t)})&=\sigma_{\max}(\mata_{k}^{(t)}(\mata_{k}^*)^{-1})\leq\sigma_{\max}(\mata_{k}^{(t)})\sigma_{\max}((\mata_{k}^*)^{-1})\\
            &\leq\sigma_{\max}((\mata_{k}^*)^{-1})C^{d+1-k},\\
			\sigma_{\min}(\tilde{\mata}_k^{(t)})&\geq\frac{\sigma_{\min}(\mata_{k}^{(t)})}{\sigma_{\max}(\mata_{k}^*)}\geq\frac{1}{C^{d+k-1}\sqrt{r_1}\sigma_{\max}(\mata_{k}^*)}>0,
		\end{aligned}
	\end{equation*}
	where we use the fact that $\sigma_{\min}(\tilde{\mata}_k^{(t)})\sigma_{\max}(\mata_k^*)\geq\sigma_{\min}(\tilde{\mata}_k^{(t)}\mata_k^*)=\sigma_{\min}(\mata_k^{(t)})$.}}

	{{\color{black}Subsequently, we}} prove that the auxiliary sequence $\{\tilde{\mata}_k^{(t)}\}_{t=0}^{\infty}$ converges to $c_k^{-1}\matI_{r_k}$ for some $c_k\neq 0$ {{\color{black}by considering the diagonal and the off-diagonal elements of $(\tilde{\mata}_{k+1}^{(t)})^{-1}$}}. Consider the off-diagonal elements of $(\tilde{\mata}_{k+1}^{(t)})^{-1}$ and recall that $(\tilde{\mata}_{k+1}^{(t)})^{-1}\otimes(\tilde{\mata}_k^{(t)})^\top$ converges to $\matI_{r_kr_{k+1}}$, i.e., $((\tilde{\mata}_{k+1}^{(t)})^{-1})_{i,j}(\tilde{\mata}_k^{(t)})^\top$ converges to $0$ if $i\neq j$ or $\matI_{r_{k}}$ if $i=j$. Since {{\color{black}$\|\tilde{\mata}_{k}^{(t)}\|_\frob\geq\sigma_{\min}(\tilde{\mata}_k^{(t)})\geq1/(C^{d+k-1}\sqrt{r_1}\sigma_{\max}(\mata_{k}^*))>0$}}, it holds that $\lim_{t\to\infty}((\tilde{\mata}_{k+1}^{(t)})^{-1})_{i,j}=0$ for $i\neq j$. 
	
	{{\color{black}Then, we consider the diagonal elements of $(\tilde{\mata}_{k+1}^{(t)})^{-1}$. By using 
    \[\lim_{t\to\infty}((\tilde{\mata}_{k+1}^{(t)})^{-1})_{i,i}(\tilde{\mata}_k^{(t)})^\top-((\tilde{\mata}_{k+1}^{(t)})^{-1})_{j,j}(\tilde{\mata}_k^{(t)})^\top=\matI_{r_k}-\matI_{r_k}=0\] 
    for $i\neq j$, we also obtain that $((\tilde{\mata}_{k+1}^{(t)})^{-1})_{i,i}-((\tilde{\mata}_{k+1}^{(t)})^{-1})_{j,j}$ converge to $0$. Therefore, there exists $c_{k+1}^{(t)}\neq 0$ such that $(\tilde{\mata}_{k+1}^{(t)})^{-1}-c_{k+1}^{(t)}\matI_{r_{k+1}}$ converges to $0$, where $c_{k+1}^{(t)}:=((\tilde{\mata}_{k+1}^{(t)})^{-1})_{1,1}$. In other words, $\tilde{\mata}_{k+1}^{(t)}=(c_{k+1}^{(t)})^{-1}\matI_{r_{k+1}}^{}+o(1)$ holds for sufficiently large~$t$. Since $\sigma_{\min}(\tilde{\mata}_k^{(t)})$ is bounded, the sequences $\{c_k^{(t)}\}$ is precompact. It follows from $\lim_{t\to\infty}(\tilde{\mata}_{k+1}^{(t)})^{-1}\otimes(\tilde{\mata}_k^{(t)})^\top=\matI_{r_kr_{k+1}}$ that $\lim_{t\to\infty}c_{k+1}^{(t)}/c_k^{(t)}=1$. Therefore, it suffices to consider the sequence $\{c_1^{(t)}\}_{t=1}^{\infty}$. Recall that $\tilde{\mata}_1^{(t)}=\mata_1^{(t)}(\mata_1^*)^{-1}$ and $\|\mata_1^{(t)}\|_\frob=\|\mata_1^*\|_\frob=1$. It holds that $\mata_1^{(t)}=\mata_1^*/c_1^{(t)}+o(1)$. Therefore, by taking the Frobenius norm and choosing an appropriate gauge (multiplying $\pm 1$ on $\vecmatA^{(t)}$ if necessary), we obtain that $\lim_{t\to\infty}c_1^{(t)}=1$. Consequently, we have $\lim_{t\to\infty}c_k^{(t)}=1$ for all $k\in[d]$ and thus $\lim_{t\to\infty}\mata_k^{(t)}=\mata_{k}^*$, i.e.,}} $\{[\vecmatA^{(t)}]\}$ converges to $[\vecmatA^*]$. 
\end{proof}

\begin{remark}
    It is worth noting that the proofs of the properness for TT and TR is essentially different. For TT decomposition (or more generally, hierarchical Tucker decomposition), the properness follows directly from a sequential construction of the Moore--Penrose pseudoinverses~\cite[Lemma 1]{uschmajew2013geometry} along the chain structure. In contrast, due to the ring structure of the TR decomposition, such a sequential procedure is no longer applicable, as the gauge invariance globally couples all core tensors. As a remedy, we exploit the full-rank conditions on the core tensors and carefully control the singular values to establish the properness.
\end{remark}

As a direct result, it follows from Propositions~\ref{prop: free} and~\ref{prop: proper} and~\cite[Theorem 9.18]{boumal2023intromanifolds} that $\tensM_\vecr^*=\bartensM_\vecr^*/\PGL(\vecr)$ is a smooth quotient manifold. 
\begin{theorem}[quotient manifold]\label{thm:TRquotient}
	$\tensM_\vecr^*$ is a smooth quotient manifold of dimension 
	\[\dim(\tensM_\vecr^*)=\sum_{k=1}^{d}r_kn_kr_{k+1}-\sum_{k=1}^{d}r_k^2+1.\]
\end{theorem}

Denote the canonical projection of $\tensM_\vecr^*$ by $\pi:\bartensM_\vecr^*\to\tensM_\vecr^*$. \cref{fig: comm diagram} provides a commutative diagram of injective TR tensors, where we recall that $\tau:\bartensM_{\vecr}\to\mathbb{R}^{n_1\times n_2\times\cdots\times n_d}$ is a smooth TR mapping defined by $\tau:\vectensU\mapsto\llbracket\tensU_1,\tensU_2,\dots,\tensU_d\rrbracket$. Note that $\tau|_{\tensM_\vecr^*}=\tau\circ\pi^{-1}$ is bijective from~\cref{thm: fundamental}.

\begin{figure}[htbp]
	\centering
	\begin{tikzpicture}
		\def\x{1.6};

		\node (Mbar) at (0,\x) {$\bartensM_\vecr^*$};
		\node (M) at (0,0) {$\tensM_\vecr^*=\bartensM_\vecr^*/\PGL(\vecr)$};
		\node (N) at (3*\x,0) {$\mathbb{R}^{n_1\times n_2\times\cdots\times n_d}$};
		
		\draw[->] (Mbar) -- (M);
		\draw[->] (M) -- (N);
		\draw[->] (Mbar) -- ($(N)+(-0.95,0.25)$);
		
		\node[left] at ($0.5*(Mbar)+0.5*(M)$) {$\pi$};
		\node[below] at ($0.5*(M)+0.5*(N)$) {$\tau$};
		\node[above] at ($0.5*(Mbar)+0.5*(N)$) {$\tau\circ\pi$};
	\end{tikzpicture}
	\caption{Commutative diagram of injective TR tensors}
	\label{fig: comm diagram}
\end{figure}

\subsection{Vertical and horizontal spaces}
Recall that given $\vectensU\in\bartensM_\vecr^*$, the vertical space is defined by $\vertical_\vectensU\!\bartensM_\vecr^*=\tangent_\vectensU\!\tensF_\vectensU=\ker\mathrm{D}\pi(\vectensU)$, where $\tensF_\vectensU=\{\theta_{[\vecmatA]}(\vectensU):[\vecmatA]\in\PGL(\vecr)\}$ is the orbit of $\vectensU$ through group action $\theta$. {{\color{black}It follows from~\cite[Proposition~9.3]{boumal2023intromanifolds} that $\tensF_\vectensU$ is an embedded submanifold of $\bartensM_\vecr^*$ and thus 
\[\dim(\vertical_\vectensU\!\bartensM_\vecr^*)=\dim(\tangent_\vectensU\!\tensF_\vectensU)=\dim(\tensF_\vectensU)=\dim(\PGL(\vecr))=\sum_{k=1}^{d}r_k^2-1.\] 
The orthogonal complement of $\vertical_\vectensU\!\bartensM_\vecr^*$ with respect to the Riemannian metric $g$ is the horizontal space $\horizontal_\vectensU\!\bartensM_\vecr^*$, where the Riemannian metric $g$}}
is chosen as the Euclidean metric. A schematic illustration of quotient manifold, vertical and horizontal spaces is shown in~\cref{fig: Manifold}. {{\color{black}Note that the vertical space $\vertical_\vectensU\!\bartensM_\vecr^*$ is the set of $\vec{\eta}=(\eta_1,\eta_2,\dots,\eta_d)$ for which there exist free variables $\matd_k\in\mathbb{R}^{r_k\times r_k}$ satisfying $\tr(\matd_1)=0$ such that $\eta_k=\tensU_k\!\times_1\!\matd_k-\tensU_k\!\times_3\!\matd_{k+1}^\top$}}. 

\begin{figure}[htbp]
	\centering
	\begin{tikzpicture}[scale = 0.8]
		\coordinate (u) at (0,0); 
		\coordinate (horizontal) at (1.6,0); 
		\coordinate (vertical) at (0,2); 
		\coordinate (hs) at ($(u)-(1.2,0.5)$);
		\fill (u) circle (2.5pt);
		\node[right] at (u) {$\vectensU$};
		\draw[magenta,-,very thick] ($(u)+(-0.55,-2)$) arc (-30:30:4);
		\draw[-] ($(u)+(-0.55,-2)+(horizontal)$) arc (-30:30:4);
		\draw[-] ($(u)+(-0.55,-2)-(horizontal)$) arc (-30:30:4);
		\node[right] (Fu) at ($(u)+(-0.55,-2)$) {$\tensF_\vectensU$};
		\node[left] (m) at ($(u)-1.5*(horizontal)$) {$\bartensM_\vecr^*$};
		\draw[-] ($(u)-(vertical)$) -- ($(u)+(vertical)$);
		\node[right] at ($(u)+(vertical)$) {$\vertical_\vectensU\!\bartensM_\vecr^*$};
		\path[draw] (hs) -- ($(hs)+(2,0)$) -- ($(hs)+(2.4,1)$) -- ($(hs)+(0.4,1)$) -- cycle;
		\node[right] at ($(hs)+(2,0)$) {$\horizontal_\vectensU\!\bartensM_\vecr^*$};
		
		\node[above] at ($(u)+2*(horizontal)$) {\Large$\pi$};
		\draw[thick,->] ($(u)+1.5*(horizontal)$) -- ($(u)+2.5*(horizontal)$);
		\coordinate (piu) at ($(u)+4*(horizontal)$);
		\fill[magenta] (piu) circle (4pt);
		\node[above] at (piu) {$[\vectensU]$};
		\draw[-] ($(piu)+(2,-0.55)$) arc (60:120:4);
		\node[below] at ($(piu)+(2,-0.55)$) {$\bartensM_\vecr^*/\PGL(\vecr)$};
		\draw[-] ($(piu)+(horizontal)$) -- ($(piu)-(horizontal)$);
		\node[right] at ($(piu)+(horizontal)$) {$\tangent_{[\vectensU]}\!\tensM_\vecr^*$};
	\end{tikzpicture}
	\caption{Illustration of total manifold $\bartensM_\vecr^*$, quotient manifold $\tensM_\vecr^*=\bartensM_\vecr^*/\PGL(\vecr)$, vertical and horizontal spaces.}
	\label{fig: Manifold}
\end{figure}

\begin{proposition}\label{prop:vertical_and_horizontal}
	The vertical space and horizontal space at $\vectensU\in\bartensM_\vecr^*$ can be parametrized by
	\begin{align}
		\vertical_\vectensU\!\bartensM_\vecr^*&=\left\{\vec{\eta}\in\bartensM_\vecr\!:\!\eta_k=\tensU_k\!\times_1\!\matd_k-\tensU_k\!\times_3\!\matd_{k+1}^\top,\ k\in[d]{{\color{black},\ \tr(\matd_{1})=0,\ \matd_{d+1}=\matd_1}}\right\}\!,\label{eq: vertical}\\
		\horizontal_\vectensU\!\bartensM_\vecr^*&=\left\{\vec\xi\in\bartensM_\vecr\!:\!
			(\xi_k)_{(1)}(\tensU_k)_{(1)}^\top=(\tensU_{k-1})_{(3)}(\xi_{k-1})_{(3)}^\top\ \text{for}\ k\in[d]
		\right\},\label{eq: horizontal}
	\end{align}
	where $\tensU_0:=\tensU_d$ and $\xi_0:=\xi_d$. 
\end{proposition}
\begin{proof}
	First, we aim to characterize the vertical space. 
	Denote $V=\{\vec\eta\in\bartensM_\vecr:\eta_k=\tensU_k\times_1\matd_k-\tensU_k\times_3\matd_{k+1}^\top,\ {{\color{black}\text{for some}\ \matd_k\in\mathbb{R}^{r_k\times r_k},\ \tr(\matd_{1})=0,\ \matd_{d+1}=\matd_1}},\ k\in[d]\}$. On the one hand, consider the smooth curve $\gamma:[0,t_0)\to\tensF_\vectensU=\{\theta_{[\vecmatA]}(\tensU):[\vecmatA]\in\PGL(\vecr)\}$ defined by
	\[\gamma(t)=(\eta_1(t),\dots,\eta_d(t))\ \text{with} \ \eta_k(t)=\tensU_k\times_1(\matI_{r_k}+t\matd_k)\times_3(\matI_{r_{k+1}}+t\matd_{k+1})^{-\top}\]
	for $t\in[0,t_0)$. It is straightforward to verify that $\dot{\gamma}(0)\in V$. Therefore, $V\subseteq\vertical_\vectensU\!\bartensM_\vecr^*$. On the other hand, we aim to prove that 
	{{\color{black}$\dim(V)=\sum_{k=1}^{d} r_k^2-1=\dim(\vertical_\vectensU\!\bartensM_\vecr^*)=\dim(\tangent_\vectensU\!\tensF_\vectensU)$}}, or equivalently, $\vec\eta=0\in V$ implies that $\matd_k=0$ for $k\in[d]$. We start from $(\eta_1)_{(2)}={{\color{black}(\tensU_1)_{(2)}}}(\matI_{r_2}\otimes\matd_1-\matd_2^\top\otimes\matI_{r_1})^\top=0$. Since ${{\color{black}(\tensU_1)_{(2)}}}$ is of full-rank, we obtain that $\matI_{r_2}\otimes\matd_1=\matd_2^\top\otimes\matI_{r_1}$. Therefore, the off-diagonal elements of $\matd_2$ are zeros and thus both $\matd_1$ and $\matd_2$ are diagonal matrices. Furthermore, it follows from $\matI_{r_2}\otimes\matd_1=\matd_2^\top\otimes\matI_{r_1}$ that $\matd_1=\matd_{2}(j,j)\matI_{r_1}$ holds for $j\in[r_2]$. Since $\tr(\matd_{1})=0$, we have $\matd_{2}(j,j)=0$ and thus $\matd_1=0$ and $\matd_2=0$. Consequently, we obtain that $\matd_k=0$ recursively by using the fact that $(\tensU_k)_{(2)}$ is of full rank.  
	
	Subsequently, we characterize the horizontal space. Denote the right hand side of~\cref{eq: horizontal} by $H$. In fact, for all $\vec\xi\in\bartensM_\vecr$ and $\vec\eta\in\vertical_\vectensU\!\bartensM_\vecr^*$, we have 
	\begin{equation*}
		\begin{aligned}
			\langle\vec\xi,\vec\eta\rangle&=\sum_{k=1}^{d}\langle\xi_k,\tensU_k\times_1\matd_k-\tensU_k\times_3\matd_{k+1}^\top\rangle\\
			&=\sum_{k=1}^{d}\langle(\xi_k)_{(1)},\matd_k(\tensU_k)_{(1)}\rangle-\langle(\xi_k)_{(3)},\matd_{k+1}^\top(\tensU_k)_{(3)}\rangle\\
			&=\sum_{k=1}^{d}\langle(\xi_k)_{(1)}(\tensU_k)_{(1)}^\top-(\tensU_{k-1})_{(3)}(\xi_{k-1})_{(3)}^\top,\matd_k\rangle.
		\end{aligned}
	\end{equation*} 
	On the one hand, if $\vec\xi\in H$, we have $\langle\vec\xi,\vec\eta\rangle=0$
	holds for all $\matd_k\in\mathbb{R}^{r_k\times r_k}$ and $\tr(\matd_1)=0$. Therefore, $H\subseteq\vertical_\vectensU^\perp\!\bartensM_\vecr^*=\horizontal_\vectensU\!\bartensM_\vecr^*$. On the other hand, for all $\vec\xi\in\horizontal_\vectensU\!\bartensM_\vecr^*$, we obtain that 
	\begin{equation*}
		\begin{aligned}
			(\xi_k)_{(1)}(\tensU_k)_{(1)}^\top&=(\tensU_{k-1})_{(3)}(\xi_{k-1})_{(3)}^\top\ \text{for}\ k=2,3,\dots,d,\\
			(\xi_1)_{(1)}(\tensU_1)_{(1)}^\top&=(\tensU_{d})_{(3)}(\xi_{d})_{(3)}^\top+c\matI_{r_1},\ c\in\mathbb{R},
		\end{aligned}
	\end{equation*}
	Furthermore, we observe that 
	\[\langle\xi_k,\tensU_k\rangle=\tr((\xi_k)_{(1)}(\tensU_k)_{(1)}^\top)=\tr((\tensU_{k-1})_{(3)}(\xi_{k-1})_{(3)}^\top)=\langle\xi_{k-1},\tensU_{k-1}\rangle\]
	holds for $k=2,3,\dots,d$. Therefore, we have $c=0$ and thus $\horizontal_\vectensU\!\bartensM_\vecr^*\subseteq H$. Finally, the horizontal space can be characterized by~\cref{eq: horizontal}. 
\end{proof}

It is worth noting that despite there being $\sum_{k=1}^d r_k^2$ equations in the parametrization of $\horizontal_\vectensU\!\bartensM_\vecr^*\subseteq H$ with respect to $\xi$ in~\cref{eq: horizontal}, the redundancy can be detected. Specifically, we compute the trace of both sides of $(\xi_k)_{(1)}(\tensU_k)_{(1)}^\top=(\tensU_{k-1})_{(3)}(\xi_{k-1})_{(3)}^\top$ and yield $d$ equations $\tr((\xi_k)_{(1)}(\tensU_k)_{(1)}^\top)=\tr((\tensU_{k-1})_{(3)}(\xi_{k-1})_{(3)}^\top)$ for $k\in[d]$, where the $d$-th equation $\tr((\xi_d)_{(1)}(\tensU_d)_{(1)}^\top)=\tr((\tensU_{d-1})_{(3)}(\xi_{d-1})_{(3)}^\top)$ is redundant as it is straightforwardly implied by the other $(d-1)$ equations. Consequently, there are $\sum_{k=1}^d r_k^2-1$ linearly independent equations in~\cref{eq: horizontal}. Such redundancy appears from the ring structure of TR.  

In the light of~\cref{prop:vertical_and_horizontal}, we provide the orthogonal projection of a vector $\vectensV\in\bartensM_\vecr$ onto the vertical and horizontal spaces.
\begin{proposition}\label{eq:v_h_TR}
	Given $\vectensV\in\bartensM_\vecr$ and $\vectensU\in\bartensM_\vecr^*$, the orthogonal projection of $\vectensV$ onto $\vertical_\vectensU\!\bartensM_\vecr^*$ and $\horizontal_\vectensU\!\bartensM_\vecr^*$ can be computed by
	\begin{equation*}
		\begin{aligned}
			\proj_{\vertical_\vectensU\!\bartensM_\vecr^*}\vectensV&=(\tensU_1\times_1\matd_1-\tensU_1\times_3\matd_2^\top,\dots,\tensU_d\times_1\matd_d-\tensU_d\times_3\matd_1^\top),\\
			\proj_{\horizontal_\vectensU\!\bartensM_\vecr^*}\vectensV&=\vectensV-\proj_{\vertical_\vectensU\!\bartensM_\vecr^*}\vectensV,
		\end{aligned}
	\end{equation*}
	where $\matd_k\in\mathbb{R}^{r_k\times r_k}$ is determined by the matrix equations
	\begin{equation}\label{eq: projection}
		\begin{aligned}
			\begin{bmatrix}
				\mata_1 & \matb_1 &  &  & \matb_d^\top\\
				\matb_1^\top & \mata_2 & \matb_2 & & \\
				& \ddots & \ddots & \ddots & \\
				& & \matb_{d-2}^\top & \mata_{d-1} & \matb_{d-1}\\
				\matb_d & & & \matb_{d-1}^\top & \mata_d
			\end{bmatrix}
			\begin{bmatrix}
				\rmvec(\matd_1)\\
				\rmvec(\matd_2)\\
				\vdots\\
				\rmvec(\matd_{d-1})\\
				\rmvec(\matd_d)
			\end{bmatrix}&=
			\begin{bmatrix}
				\vecb_1\\
				\vecb_2\\
				\vdots\\
				\vecb_{d-1}\\
				\vecb_d
			\end{bmatrix},\\
			\tr(\matd_{1})&=0,
		\end{aligned}
	\end{equation}
	where 
    \begin{align*}
        \mata_k&=\matI_{r_k}\otimes((\tensU_{k-1})_{(3)}^{}(\tensU_{k-1})_{(3)}^\top)+((\tensU_k)_{(1)}^{}(\tensU_k)_{(1)}^\top)\otimes\matI_{r_k},\\
        \matb_k&=-((\tensU_k)_{(1)}\otimes\matI_{r_k})(\matI_{r_{k+1}}\otimes(\tensU_k)_{(3)}^\top),\\
        \vecb_k&=\rmvec((\tensV_k)_{(1)}(\tensU_k)_{(1)}^\top-(\tensU_{k-1})_{(3)}(\tensV_{k-1})_{(3)}^\top)
    \end{align*}
    for $k\in[d]$. Moreover, the linear system~\cref{eq: projection} has a unique solution.
\end{proposition}
\begin{proof}
	Since $\proj_{\vertical_\vectensU\!\bartensM_\vecr^*}\!\vectensV\in\vertical_\vectensU\!\bartensM_\vecr^*$, the projection can be represented by $\proj_{\vertical_\vectensU\!\bartensM_\vecr^*}\!\vectensV=(\tensU_1\times_1\matd_1-\tensU_1\times_3\matd_2^\top,\dots,\tensU_d\times_1\matd_d-\tensU_d\times_3\matd_1^\top)$ with $\matd_k\in\mathbb{R}^{r_k\times r_k}$, 
	{{\color{black}where the matrices~$\matd_k$ can be solved by the following constrained least-squares problem:
	\begin{equation}
		\label{eq:orthproj_optim}
		\begin{aligned}
		\min_{\matd_k\in\mathbb{R}^{r_k\times r_k}}\ &\frac12\|\vectensV-(\tensU_1\times_1\matd_1-\tensU_1\times_3\matd_2^\top,\dots,\tensU_d\times_1\matd_d-\tensU_d\times_3\matd_1^\top)\|_\frob^2\\
		\subjectto\qquad &\tr(\matd_{1})=0.
		\end{aligned}
	\end{equation}

	We derive the optimality conditions of~\cref{eq:orthproj_optim}. The Lagrangian function of~\cref{eq:orthproj_optim} is given by
	\begin{equation*}
		\mathcal{L}(\vec{\matd},\lambda)=\frac12\|\vectensV-(\tensU_1\times_1\matd_1-\tensU_1\times_3\matd_2^\top,\dots,\tensU_d\times_1\matd_d-\tensU_d\times_3\matd_1^\top)\|_\frob^2+\lambda\tr(\matd_1),
	\end{equation*}
	where $\vec{\matd}:=(\matd_1,\dots,\matd_d)$. The KKT conditions of~\cref{eq:orthproj_optim} are
	\begin{equation*}
		\mathrm{D}_{\vec{\matd}}\mathcal{L}(\vec{\matd},\lambda)[(\dot{\matd}_1,\dots,\dot{\matd}_d)]
		=-\sum_{k=1}^{d}\langle\tensR_k,\tensU_k\times_1\dot{\matd}_k-\tensU_k\times_3\dot{\matd}_{k+1}^\top\rangle+\lambda\tr(\dot{\matd}_1)=0
	\end{equation*}
	for all $(\dot{\matd}_1,\dots,\dot{\matd}_d)\in\prod_{k=1}^d\mathbb{R}^{r_k\times r_k}$, together with $\tr(\matd_1)=0$,
	where $\tensR_k:=\tensV_k-\tensU_k\times_1\matd_k+\tensU_k\times_3\matd_{k+1}^\top$ is the residual tensor. 

	Then, we show that $\lambda=0$ by letting $\dot{\matd}_k=\matI_{r_k}$ for $k\in[d]$ in the above equation. We have
	\begin{equation*}
		\begin{aligned}
			0=-\sum_{k=1}^{d}\langle\tensR_k,\tensU_k\times_1\matI_{r_k}-\tensU_k\times_3\matI_{r_{k+1}}\rangle+\lambda\tr(\matI_{r_1})=\lambda r_1,
		\end{aligned}
	\end{equation*}
	which implies $\lambda=0$. Therefore, the Lagrange multiplier vanishes and the optimality conditions of~\cref{eq:orthproj_optim} are equivalent to the following equations
	\begin{equation*}
		\sum_{k=1}^{d}\langle\tensR_k,\tensU_k\times_1\dot{\matd}_k-\tensU_k\times_3\dot{\matd}_{k+1}^\top\rangle=0,\quad \tr(\matd_1)=0
	\end{equation*}
	for all $(\dot{\matd}_1,\dots,\dot{\matd}_d)\in\prod_{k=1}^d\mathbb{R}^{r_k\times r_k}$. The above equations can also be interpreted as the orthogonality condition 
	\[(\tensR_1,\dots,\tensR_d)=\vectensV-\proj_{\vertical_\vectensU\!\bartensM_\vecr^*}\vectensV\in(\vertical_\vectensU\!\bartensM_\vecr^*)^\perp=\horizontal_\vectensU\!\bartensM_\vecr^*.\] Therefore, it follows from~\cref{eq: horizontal} that
	}}
	\begin{equation*}
		\begin{aligned}
			&((\tensV_k)_{(1)}-\matd_k(\tensU_k)_{(1)}+(\tensU_k)_{(1)}(\matd_{k+1}\otimes\matI_{n_k}))(\tensU_k)_{(1)}^\top\\
			&\qquad=(\tensU_{k-1})_{(3)}((\tensV_{k-1})_{(3)}-(\tensU_{k-1})_{(3)}(\matI_{n_{k-1}}\otimes\matd_{k-1}^\top)+\matd_{k}^\top(\tensU_{k-1})_{(3)})^\top,
		\end{aligned}
	\end{equation*}
	By vectorizing the left-hand side, we obtain
	\begin{equation*}
		\begin{aligned}
			&~~~~\rmvec(((\tensV_k)_{(1)}-\matd_k(\tensU_k)_{(1)}+(\tensU_k)_{(1)}(\matd_{k+1}\otimes\matI_{n_k}))(\tensU_k)_{(1)}^\top)\\
			&=\rmvec((\tensV_k)_{(1)}(\tensU_k)_{(1)}^\top)-((\tensU_k)_{(1)}(\tensU_k)_{(1)}^\top\otimes\matI_{r_k})\rmvec(\matd_k)\\
            &\qquad\qquad+\rmvec((\tensU_k)_{(1)}(\matd_{k+1}\otimes\matI_{n_k})(\tensU_k)_{(1)}^\top)\\
			&=\rmvec((\tensV_k)_{(1)}(\tensU_k)_{(1)}^\top)-((\tensU_k)_{(1)}(\tensU_k)_{(1)}^\top\otimes\matI_{r_k})\rmvec(\matd_k)\\
            &\qquad\qquad+((\tensU_k)_{(1)}\otimes\matI_{r_k})\rmvec(\tensU_k\times_3\matd_{k+1}^\top)\\
			&=\rmvec((\tensV_k)_{(1)}(\tensU_k)_{(1)}^\top)-((\tensU_k)_{(1)}(\tensU_k)_{(1)}^\top\otimes\matI_{r_k})\rmvec(\matd_k)\\
            &\qquad\qquad+((\tensU_k)_{(1)}\otimes\matI_{r_k})(\matI_{r_{k+1}}\otimes(\tensU_k)_{(3)}^\top)\rmvec(\matd_{k+1}).
		\end{aligned}
	\end{equation*}
	Similarly, vectorizing the right-hand side gives
	\begin{equation*}
		\begin{aligned}
			&~\rmvec((\tensU_{k-1})_{(3)}((\tensV_{k-1})_{(3)}-(\tensU_{k-1})_{(3)}(\matI_{n_{k-1}}\otimes\matd_{k-1}^\top)+\matd_{k}^\top(\tensU_{k-1})_{(3)})^\top)\\
			=&~\rmvec((\tensU_{k-1})_{(3)}(\tensV_{k-1})_{(3)}^\top)-(\matI_{r_{k-1}}\otimes(\tensU_{k-1})_{(3)})((\tensU_{k-1})_{(1)}^\top\otimes\matI_{r_{k-1}})\rmvec(\matd_{k-1})\\
            &\qquad\qquad+(\matI_{r_k}\otimes(\tensU_{k-1})_{(3)}(\tensU_{k-1})_{(3)}^\top)\rmvec(\matd_k).
		\end{aligned}
	\end{equation*}
	{{\color{black}Since the Lagrange multiplier vanishes, the vectorized equations together with the constraint $\tr(\matd_1)=0$ form the rectangular system~\cref{eq: projection}. Hence, the solutions of~\cref{eq: projection} are the first-order stationary points of~\cref{eq:orthproj_optim}.
	
	Subsequently, we aim to prove that the solution of~\cref{eq: projection} exists and is unique by strong convexity of~\cref{eq:orthproj_optim} under the constraint $\tr(\matd_1)=0$. In fact, the proof of~\cref{prop:vertical_and_horizontal} shows that the mapping 
	\[L_\vectensU:\vec{\matd}\mapsto(\tensU_1\times_1\matd_1-\tensU_1\times_3\matd_2^\top,\dots,\tensU_d\times_1\matd_d-\tensU_d\times_3\matd_1^\top)\] 
	from the parameter space $D_0:=\{\vec{\matd}\in\prod_{k=1}^d\mathbb{R}^{r_k\times r_k}:\tr(\matd_1)=0\}$ to the vertical space $\vertical_\vectensU\!\bartensM_\vecr^*$ is injective. Therefore, the objective function $\|\vectensV-L_\vectensU(\vec{\matd})\|_\frob^2/2$ of~\cref{eq:orthproj_optim} is a strongly convex quadratic function on $D_0$ and hence admits a unique minimizer.}}
\end{proof}

{{\color{black}
	\begin{remark}
		As discussed after the characterization of the horizontal space, the ring structure of the TR decomposition introduces one redundant equation in the linear system~\cref{eq: projection}. In practice, by computing matrices $\mata_k,\matb_k,\vecb_k$ and concatenating $\matd_1,\matd_2,\dots,\matd_d$, we assemble the symmetric matrix in~\cref{eq: projection}. Since the last scalar equation is redundant owing to the ring structure, we replace the last row of the linear system by the equation $\tr(\matd_1)=0$. The resulting system is currently solved by MATLAB's backslash operator. 

		Denote $n:=\max_{k\in[d]}n_k$, and $r:=\max_{k\in[d]}r_k$. The linear system in~\cref{eq: projection} contains $\mathcal{O}(dr^2)$ undetermined variables. Solving the linear system consists of three steps: 1) computing the coefficient blocks $\mata_k$, $\matb_k$ and $\vecb_k$, which requires $\mathcal{O}(6dnr^4)$ flops by using the Kronecker product structure of $\matb_k$; 2) solving the linear system~\cref{eq: projection}, which involves $\mathcal{O}(d^3r^6)$ flops; 3) once the matrices $\matd_k$ are obtained, constructing the vertical and horizontal projections requires $\mathcal{O}(4dnr^3)$ flops. Arranging the terms in descending powers of \(r\), the overall computational cost of the projections is $\mathcal{O}(d^3r^6+6dnr^4+4dnr^3)=\mathcal{O}(d^3r^6 + dnr^4)$ flops.
	\end{remark}
}}

\subsection{{{\color{black}Immersed submanifold}}}
{{\color{black}
A natural question is whether the image $\tau(\tensM_\vecr^*)$ of the quotient manifold $\tensM_\vecr^*$ under the TR mapping $\tau$ enjoys a manifold structure. We show that $\tau(\tensM_\vecr^*)$ is an immersed submanifold of $\mathbb{R}^{n_1\times n_2\times\cdots\times n_d}$ by proving $\ker(\mathrm{D}\tau(\vectensU))=\vertical_{\vectensU}\!\bartensM_{\vecr}^*$. To this end, we first establish the following infinitesimal gauge-transformation lemma.}}

{{\color{black}
\begin{lemma}\label{lem:gauge_trans}
	Let $L_k:\mathbb{R}^{r_k\times r_{k+1}}\to\mathbb{R}^{r_k\times r_{k+1}}$ be a linear mapping for $k\in[d]$, where $d\geq 3$, $r_k\geq 2$, and $r_{d+1}:=r_1$. If
	\begin{equation}
		\label{eq:sum_of_traces}
		\sum_{k=1}^d\tr(\mata_1\cdots\mata_{k-1} L_k(\mata_k)\mata_{k+1}\cdots\mata_d)=0
	\end{equation} 
	holds for all $\mata_k\in\mathbb{R}^{r_k\times r_{k+1}}$, there exists $\matx_k\in\mathbb{R}^{r_k\times r_k}$ such that 
	\[L_k:\mata\mapsto -\matx_k\mata+\mata\matx_{k+1}.\]
\end{lemma} 
\begin{proof}
	Let $\{\mate_{i,j}^{(k)}\}$ denote the standard basis of $\mathbb{R}^{r_k\times r_{k+1}}$, where the $(i,j)$-th entry of $\mate_{i,j}^{(k)}$ is one and all other entries are zero. It holds that 
	\[L_k(\mate_{i,j}^{(k)})=\sum_{i^\prime=1}^{r_k}\sum_{j^\prime=1}^{r_{k+1}}\ell_{i,j; i^\prime,j^\prime}^{(k)}\mate_{i^\prime,j^\prime}^{(k)}.\] 
	
	We prove the result by characterizing the coefficients $\ell_{i,j; i^\prime,j^\prime}^{(k)}$. Let $\mata_k=\mate_{i_k,j_k}^{(k)}$ and define $S:=\{k:j_k\neq i_{k+1}\}$, where all indices are understood cyclically, e.g., $i_{d+1}:=i_1$. Since $\tr(\mata_1\cdots\mata_{k-1} L_k(\mata_k)\mata_{k+1}\cdots\mata_d)\neq 0$ implies that $S\subseteq\{k-1,k\}$, and every summand in~\cref{eq:sum_of_traces} vanishes when $|S|\geq 3$, we consider the following three cases $S=\{k-1,k\}$, $S=\{k\}$, and $S=\emptyset$ to simplify the representation of $L_k(\mate_{i,j}^{(k)})$. 
	
	First, fix $k\in[d]$ and choose the indices $\{i_k\}$ and $\{j_k\}$ such that $S=\{k-1,k\}$. Let $\mata_k=\mate^{(k)}_{i_k,j_k}$. It follows from~\cref{eq:sum_of_traces} that $\tr(\mata_1\cdots\mata_{k-1} L_k(\mata_k)\mata_{k+1}\cdots\mata_d)=0$ and thus $\langle\mate^{(k)}_{j_{k-1},i_{k+1}},L_k(\mata_{k})\rangle=0$. Therefore, we have $\ell_{i_k,j_k; j_{k-1},i_{k+1}}^{(k)}=0$, or equivalently, $\ell_{i,j; i^\prime,j^\prime}^{(k)}=0$ for all $i^\prime\in[r_k]\setminus\{i\}$ and $j^\prime\in[r_{k+1}]\setminus\{j\}$. Hence, the representation of $L_k(\mate_{i,j}^{(k)})$ simplifies to
	\begin{equation}
		\label{eq:First_simp}
		\begin{aligned}
			L_k(\mate_{i,j}^{(k)})&=\sum_{i^\prime=1, i^\prime\neq i}^{r_k}\ell_{i,j; i^\prime,j}^{(k)}\mate_{i^\prime,j}^{(k)} + \sum_{j^\prime=1, j^\prime\neq j}^{r_{k+1}}\ell_{i,j; i,j^\prime}^{(k)}\mate_{i,j^\prime}^{(k)}+\ell_{i,j; i,j}^{(k)}\mate_{i,j}^{(k)},
		\end{aligned}
	\end{equation}
	where the three terms correspond to changes in the row index, changes in the column index, and diagonal scaling respectively. 

	In the next two steps, we aim to identify the off-diagonal part of $\matx_k$ from the first and second summands, and show that the remaining summand can be split into the diagonal entries of $\matx_k$ and $\matx_{k+1}$.
	
	Second, fix $k\in[d]$ and choose the indices $\{i_k\}$ and $\{j_k\}$ such that $S=\{k\}$. Let $\mata_k=\mate^{(k)}_{i_k,j_k}$. The only possibly nonzero summands in~\cref{eq:sum_of_traces} are the $k$-th and $(k+1)$-th summands. Therefore,
	\[\tr(\mata_1\cdots\mata_{k-1} L_k(\mata_k)\mata_{k+1}\cdots\mata_d)+\tr(\mata_1\cdots\mata_{k} L_{k+1}(\mata_{k+1})\mata_{k+2}\cdots\mata_d)=0.\] 
	i.e., $\ell_{i_k,j_k; j_{k-1},i_{k+1}}^{(k)}+\ell_{i_{k+1},j_{k+1}; j_{k},i_{k+2}}^{(k+1)}=0$ or $\ell_{i_k,j_k; i_k,i_{k+1}}^{(k)}+\ell_{i_{k+1},i_{k+2}; j_{k},i_{k+2}}^{(k+1)}=0$
	holds for all $i_k\in[r_k]$, $j_k,i_{k+1}\in[r_{k+1}]$ with $j_k\neq i_{k+1}$, and $i_{k+2}\in[r_{k+2}]$. Denote 
	\[x_{i,j}^{(k+1)}:=\ell_{i_k,i; i_k,j}^{(k)}=-\ell_{j,i_{k+2}; i,i_{k+2}}^{(k+1)}\]
	for $i,j\in[r_{k+1}]$ and $i\neq j$. Since the identity holds for all $i_k\in[r_k]$ and $i_{k+2}\in[r_{k+2}]$ and $r_k,r_{k+2}\geq 2$, $x_{i,j}^{(k+1)}$ is well-defined. Therefore,
	\[L_k(\mate_{i,j}^{(k)})=-\tilde{\matx}_k\mate_{i,j}^{(k)}+\mate_{i,j}^{(k)}\tilde{\matx}_{k+1}+\ell_{i,j; i,j}^{(k)}\mate_{i,j}^{(k)},\]
	where $(\tilde{\matx}_k)_{i,j}=x_{i,j}^{(k)}$ if $i\neq j$ and $(\tilde{\matx}_k)_{i,j}=0$ otherwise. Note that this step shows that the off-diagonal entries of $\tilde{\matx}_{k+1}$ determined from $L_k$ are consistent with those from $L_{k+1}$.	
	
	Third, to absorb the coefficients $\ell_{i,j; i,j}^{(k)}$ into the diagonals of $\tilde{\matx}_k$ and $\tilde{\matx}_{k+1}$, we aim to prove the identity
	\[\sum_{k=1}^d\ell_{i_k,i_{k+1};i_k,i_{k+1}}^{(k)}=0\]
	for all $i_k\in[r_k]$ as follows. Denote $\lambda^{(k)}_{i,j}:=\ell_{i,j;i,j}^{(k)}$. For indices $\{i_k\}$ and $\{j_k\}$ satisfying $S=\emptyset$, let $\mata_k=\mate^{(k)}_{i_k,j_k}$. Comparing the equalities obtained from setting $i_k=u$ and $i_k=v$, we have
	\[\lambda_{i_{k-1},u}^{(k-1)}-\lambda_{i_{k-1},v}^{(k-1)}=\lambda_{v,i_{k+1}}^{(k)}-\lambda_{u,i_{k+1}}^{(k)}.\]
	Since $r_{k-1},r_{k+1}\geq 2$, the left-hand side is independent of $i_{k+1}$ and the right-hand side is independent of the choice of $i_{k-1}$. Hence, there exists a mapping $\phi$ such that $\phi(u,v)=\lambda_{v,i_{k+1}}^{(k)}-\lambda_{u,i_{k+1}}^{(k)}$, i.e., $\phi(u,v)+\lambda_{u,i_{k+1}}^{(k)}=\lambda_{v,i_{k+1}}^{(k)}$, which is independent of $u$. For a fixed $u_0$, let $p_k(a):=\phi(u_0,a)$ and $q_k(b):=\lambda^{(k)}_{u_0,b}$. It holds that
	\[\lambda_{i,j}^{(k)}=p_k(i)+q_k(j).\]
	Therefore, we have
	\begin{align*}
		\sum_{k=1}^d\ell_{i_{k},i_{k+1}; i_k,i_{k+1}}^{(k)} &= \sum_{k=1}^d (p_k(i_k)+q_k(i_{k+1}))= \sum_{k=1}^d (p_k(i_k)+q_{k-1}(i_{k}))= 0
	\end{align*}
	holds for all $i_k\in[r_k]$. Therefore, there exists a constant $c_k$, such that $c_k=p_k(a)+q_{k-1}(a)$ for any $a\in[r_k]$. Since $\sum_{k=1}^{d}c_k=0$, there exists $s_1,\dots,s_d$, such that $c_k=s_k-s_{k-1}$. Consequently, we have
	\begin{align*}
		\lambda^{(k)}_{i,j}&=p_k(i)+q_k(j)=p_k(i)+c_{k+1}-p_{k+1}(j)\\
		&=p_k(i)-s_{k}-p_{k+1}(j)+s_{k+1}.
	\end{align*}
	
	Let $\matx_k:=\tilde{\matx}_k+\diag(-p_k(1)+s_k,\dots,-p_k(r_k)+s_k)$. It follows that
	\[L_k(\mate_{i,j}^{(k)})=-\matx_k\mate_{i,j}^{(k)}+\mate_{i,j}^{(k)}\matx_{k+1}.\]
	
\end{proof}
}}

{{\color{black}
\begin{remark}
	Since we assume $r_k\geq 2$ for all $k\in[d]$, the proof of~\cref{lem:gauge_trans} is not applicable to the tensor train case, where at least one rank parameter equals one. In fact, the proof for the tensor ring format is essentially different from the sequential identification for tree tensor networks in~\cite{uschmajew2013geometry}. The identification exploits the tree structure by starting from a boundary core and identifying the gauge matrices explicitly and recursively. In contrast, tensor ring has a cyclic structure. One may choose a starting index $k$ and construct $\matx_k$ and $\matx_{k+1}$ from~\eqref{eq:First_simp}; however, one must additionally verify that the matrix $\matx_k$ constructed from the $k$-th equation coincides with the one from the $(k-1)$-th equation. This consistency is established in the second and third steps of the proof.
\end{remark}	
}}

{{\color{black}Next, we aim to show that $\tau$ is an injective immersion, i.e., $\ker(\mathrm{D}\tau(\vectensU))=\vertical_\vectensU\!\bartensM_\vecr^*$ by using~\cref{lem:gauge_trans}. 
\begin{theorem}
	Let $d\geq3$ and $r_k\geq2$ for all $k\in[d]$. Given $\vectensU\in\bartensM_\vecr^*$, the mapping $\tau|_{\tensM_\vecr^*}$ is an injective immersion. Furthermore, $\tau(\tensM_\vecr^*)$ is an immersed submanifold of $\mathbb{R}^{n_1\times n_2\times\cdots\times n_d}$.
\end{theorem}
\begin{proof}
	It suffices to prove that if the kernel condition $\mathrm{D}\tau(\vectensU)[\dot{\tensU}_1,\dot{\tensU}_2,\dots,\dot{\tensU}_d]=0$ holds for $(\dot{\tensU}_1,\dot{\tensU}_2,\dots,\dot{\tensU}_d)\in\bartensM_{\vecr}^*$, then $(\dot{\tensU}_1,\dot{\tensU}_2,\dots,\dot{\tensU}_d)\in\vertical_\vectensU\!\bartensM_\vecr^*$. We first show that the kernel condition implies~\cref{eq:sum_of_traces}.
	
	First, we show that the columns of $(\dot{\tensU}_k)_{(2)}$ lie in the column space of $(\tensU_k)_{(2)}$. Denote $(\tensU_k)_{(2)}$ and $(\dot{\tensU}_k)_{(2)}$ by $\matW_k$ and $\dot{\matW}_k$, respectively. Let $\dot{\matW}_{\neq k}$ be the mode-2 unfolding obtained by differentiating the contracted cores other than the $k$-th core. The mode-$k$ unfolding of the kernel condition is
	\[\dot{\matW}_k^{}\matW_{\neq k}^{\top}+\matW_k^{}\dot{\matW}_{\neq k}^{\top}=0.\]
	Let $\matp_k^\perp:=\matI_{n_k}-\matW_k^{}\matW_k^\dagger$ be the orthogonal projector onto the orthogonal complement of the column space of $\matW_k$. Left-multiplying the equality by $\matp_k^\perp$ gives $\matp_k^\perp\dot{\matW}_k^{}\matW_{\neq k}^{\top}=0$. It follows from~\cref{lem: rank} that $\matW_{\neq k}$ has full column rank. Therefore, $\matp_k^\perp\dot{\matW}_k=0$, and $\dot{\matW}_k$ belongs to the column space of $\matW_k$. Thus, there exists $\matR_k\in\mathbb{R}^{(r_kr_{k+1})\times(r_kr_{k+1})}$ such that $\dot{\matW}_k=\matW_k\matR_k$. Define the linear mapping $L_k:\mathbb{R}^{r_k\times r_{k+1}}\to\mathbb{R}^{r_k\times r_{k+1}}$ by 
	\[L_k:\matu\mapsto\reshape(\matR_k^\top\rmvec(\matu)).\] 
	It follows that $\dot{\tensU}_k(i_k)=L_k(\tensU_k(i_k))$. Thus, the tuple $(\dot{\tensU}_1,\dot{\tensU}_2,\dots,\dot{\tensU}_d)$ determines the linear mappings $L_1,L_2,\dots,L_d$. For the sake of brevity, $L_k(\tensU_k)$ denotes the tensor obtained by applying $L_k$ to every slice matrix of $\tensU_k$.

	Second, we use the full-rank condition to show that $\mathrm{D}\tau(\vectensU)[L_1(\tensU_1),\dots,L_d(\tensU_d)]=0$ implies~\cref{eq:sum_of_traces}. Since $(\tensU_k)_{(2)}$ has full column rank, the slice matrices of $\tensU_k$ span $\mathbb{R}^{r_k\times r_{k+1}}$. Therefore, for every $\mata_k\in\mathbb{R}^{r_k\times r_{k+1}}$, there exist coefficients $c_{i_k}^{(k)}$ such that $\mata_k=\sum_{i_k=1}^{n_k}c_{i_k}^{(k)}\matu_k(i_k)$. Taking into the left-hand side of~\cref{eq:sum_of_traces} and using the linearity of $L_k$, we obtain
	\begin{align*}
		&\sum_{k=1}^d\tr(\mata_1\cdots\mata_{k-1} L_k(\mata_k)\mata_{k+1}\cdots\mata_d)\\
		&=\sum_{k=1}^{d}\sum_{i_1=1}^{n_1}\cdots\sum_{i_d=1}^{n_d}c_{i_1}^{(1)}c_{i_2}^{(2)}\cdots c_{i_d}^{(d)}\tr(\matu_1(i_1)\cdots L_k(\matu_k(i_k))\cdots\matu_d(i_d))\\
		&=\sum_{i_1=1}^{n_1}\cdots\sum_{i_d=1}^{n_d}c_{i_1}^{(1)}c_{i_2}^{(2)}\cdots c_{i_d}^{(d)}\sum_{k=1}^{d}\tr(\matu_1(i_1)\cdots L_k(\matu_k(i_k))\cdots\matu_d(i_d))\\
		&=0.
	\end{align*}
	
	Therefore, by~\cref{lem:gauge_trans}, there exists $\matX_k\in\mathbb{R}^{r_k\times r_k}$ such that
	\[L_k:\mata\mapsto-\matX_k\mata+\mata\matX_{k+1}.\]
	Let $\matD_k:=-\matX_k$ and set $c:=\tr(\matD_1)/r_1$. Replacing every $\matD_k$ by $\matD_k-c\matI_{r_k}$ leaves the expression $\matD_k\matu_k(i_k)-\matu_k(i_k)\matD_{k+1}$ unchanged and ensures that $\tr(\matD_1)=0$. With this normalization, $\dot{\tensU}_k$ can be represented by
	\[\dot{\tensU}_k=\tensU_k\times_1\matD_k-\tensU_k\times_3\matD_{k+1}^{\top}.\]
	Hence, $(\dot{\tensU}_1,\dot{\tensU}_2,\dots,\dot{\tensU}_d)\in\vertical_\vectensU\!\bartensM_\vecr^*$, which proves $\ker(\mathrm{D}\tau(\vectensU))\subseteq\vertical_\vectensU\!\bartensM_\vecr^*$. The reverse inclusion follows directly from the telescoping cancellation. Therefore, $\ker(\mathrm{D}\tau(\vectensU))=\vertical_\vectensU\!\bartensM_\vecr^*$. The injectivity on the quotient manifold follows from~\cref{thm: fundamental}. Consequently, $\tau(\tensM_\vecr^*)$ is an immersed submanifold of $\mathbb{R}^{n_1\times n_2\times\cdots\times n_d}$.
\end{proof}
}}

{{\color{black}We finally comment on the embedded-submanifold property. A recent work~\cite{chen2025provably} provides a finite-step procedure that identifies, up to gauge invariance and with probability one, the exact TR decomposition of generic tensors in $\tau(\tensM_\vecr^*)$ for $r_1=r_2=\cdots=r_d$. This result indicates that the inverse mapping $\tau^{-1}|_{\tau(\tensM_\vecr^*)}:\tau(\tensM_\vecr^*)\to\tensM_\vecr^*$ is continuous on a generic subset. However, an embedded-submanifold result requires this inverse to be continuous on the entire image $\tau(\tensM_\vecr^*)$, so that the quotient-manifold topology agrees with the subspace topology inherited from the ambient tensor space. The available generic recovery result does not establish continuity at every injective TR tensor. This global topological issue is the remaining obstacle to proving that $\tau(\tensM_\vecr^*)$ is an embedded submanifold.}}


\subsection{Discussion: connection to tensor train decomposition}
We illustrate the connection between the geometries of tensor train decomposition~\cite{oseledets2011tensor} and the tensor ring decomposition. Specifically, the tensor train decomposition decomposes a tensor $\tensX$ into $d$ core tensors $\tensU_k\in\mathbb{R}^{s_k\times n_k\times s_{k+1}}$ for $k\in[d]$ and $s_1=s_{d+1}=1$, where the $(i_1,i_2,\dots,i_d)$-th entry of $\tensX$ is represented by the product of $d$ matrices
\[\tensX(i_1,i_2,\dots,i_d)=\matU_1(i_1)\matU_2(i_2)\cdots\matU_d(i_d),\]
see~\cref{fig: Tensor Ring} (middle). Note that we adopt the same notation as TR decomposition since TR decomposition can be viewed as a generalization of TT. We refer to the tuple $(1,s_1,s_2,\dots,s_{d-1},1)$ as the TT rank $\ranktt(\tensX)$ of $\tensX$.

\paragraph{Higher compressibility of TR} Given an injective TR tensor $\tensX=\llbracket\tensU_1,\dots,\tensU_d\rrbracket$ with $\vectensU\in\bartensM_\vecr^*$, we observe that 
\begin{equation*}
	\begin{aligned}
		\tensX(i_1,\dots,i_d)&=\tr(\prod_{j=1}^{d}\matu_j(i_j))=\rmvec(\matu_1(i_1)^\top)^\top\rmvec((\prod_{j=2}^{d-1}\matu_j(i_j))\matu_d(i_d))\\
		&=\rmvec(\matu_1(i_1)^\top)^\top(\matI_{r_1}\otimes(\prod_{j=2}^{d-1}\matu_j(i_j)))\rmvec(\matu_d(i_d))\\
		&=\rmvec(\matu_1(i_1)^\top)^\top\prod_{j=2}^{d-1}(\matI_{r_1}\otimes\matu_j(i_j))\rmvec(\matu_d(i_d))
	\end{aligned}
\end{equation*}
for $i_k\in[n_k]$, $k\in[d]$. 
{{\color{black}We can show that $\ranktt(\tensX)=(1,r_1r_2,r_1r_3,\dots,r_1r_{d},1)$. Let $s_k=r_1r_{k+1}$ for $k\in[d-1]$. Consider the core tensors $\tilde{\tensU}_1\in\mathbb{R}^{1\times n_1\times s_1}$ with $\tilde{\tensU}_1(:,i_1,:)=\rmvec(\matu_1(i_1)^\top)^\top$, $\tilde{\tensU}_k\in\mathbb{R}^{s_{k-1}\times n_k\times s_k}$ with $\tilde{\tensU}_k(:,i_k,:)=(\matI_{r_1}\otimes\matu_k(i_k))$ for $k=2,3,\dots,d-1$ and $\tilde{\tensU}_d\in\mathbb{R}^{s_{d-1}\times n_d\times 1}$ with $\tilde{\tensU}_d(:,i_d,:)=\rmvec(\matu_d(i_d))$. It holds that $\tensX=\llbracket\tilde{\tensU}_1,\tilde{\tensU}_2,\dots,\tilde{\tensU}_d\rrbracket$ is a TT decomposition of $\tensX$. 

We aim to prove that $\ranktt(\tensX)=(1,r_1r_2,r_1r_3,\dots,r_1r_{d},1)$ by computing the ranks of the mode-1 and mode-3 unfoldings of $\tilde{\tensU}_k$ for $k\in[d]$.
For $\tilde{\tensU}_1$, it is straightforward to verify that $\rank((\tilde{\tensU}_1)_{(1)})=1$ and $\rank((\tilde{\tensU}_1)_{(3)})=\rank((\tensU_1)_{(2)})=r_1r_2$. For $\tilde{\tensU}_d$, we have $\rank((\tilde{\tensU}_d)_{(1)})=\rank((\tensU_d)_{(2)})=r_1r_d$ and $\rank((\tilde{\tensU}_d)_{(3)})=1$.

For $k=2,3,\dots,d-1$, there exist permutation matrices $\tilde{\Pi}_{k,1},\tilde{\Pi}_{k,2},\tilde{\Pi}_{k,3}$ such that 
\begin{align*}
	(\tilde{\tensU}_k)_{(1)}&=[\matI_{r_1}\otimes\matu_k(1),\dots,\matI_{r_1}\otimes\matu_k(n_k)]\tilde{\Pi}_{k,1}\\
	&=\matI_{r_1}\otimes[\matu_k(1),\dots,\matu_k(n_k)]\tilde{\Pi}_{k,2}\tilde{\Pi}_{k,1}\\
	&=(\matI_{r_1}\otimes(\tensU_k)_{(1)})(\matI_{r_1}\otimes\Pi_k)^\top\tilde{\Pi}_{k,2}\tilde{\Pi}_{k,1},
\end{align*}
and 
\begin{align*}
	(\tilde{\tensU}_k)_{(3)}&=[\matI_{r_1}\otimes\matu_k(1)^\top,\dots,\matI_{r_1}\otimes\matu_k(n_k)^\top]\\
	&=\matI_{r_1}\otimes[\matu_k(1)^\top,\dots,\matu_k(n_k)^\top]\tilde{\Pi}_{k,3}\\
	&=(\matI_{r_1}\otimes(\tensU_k)_{(3)})\tilde{\Pi}_{k,3},
\end{align*}
By using~\cref{lem: rank}, we have 
\begin{align*}
	\rank((\tilde{\tensU}_k)_{(1)})&=\rank(\matI_{r_1}\otimes(\tensU_k)_{(1)})=r_1r_k,\\
	\rank((\tilde{\tensU}_k)_{(3)})&=\rank(\matI_{r_1}\otimes(\tensU_k)_{(3)})=r_1r_{k+1}.
\end{align*}

Consequently, the TT-rank of $\tensX$ is $\ranktt(\tensX)=(1,r_1r_2,r_1r_3,\dots,r_1r_{d},1)$.}} In other words, tensors in TT and TR formats have the following relationship
\[\tau(\bartensM_\vecr^*)\subseteq\{\tensX:\ranktt(\tensX)=(1,r_1r_2,r_1r_3,\dots,r_1r_{d},1)\},\]
where $\tau$ recovers a tensor from core tensors in $\bartensM_\vecr^*$ via TR decomposition. Note that if we store $\tensX$ in TT format, it requires ${\cal O}(ndr^4)$ parameters with $n=\max\{n_1,\!n_2,\dots,\!n_d\}$ and $r=\max\{r_1,r_2,\dots,r_d\}$. Nevertheless, storing $\tensX$ in TR format requires ${\cal O}(ndr^2)$ parameters. Therefore, TR decomposition is able to enjoy higher compressibility and flexibility.

\paragraph{Different quotient geometries}
We briefly introduce the quotient geometry of tensor train decomposition; see~\cite{cai2026tensor} for details. Specifically, given $\vecr=(1,r_2,r_3,\dots,r_d,1)$, the quotient manifold $\bartensM_{\vecr}^\mathrm{TT}/\GL(\vecr)$ of TT is developed from the total space $\bartensM_{\vecr}^\mathrm{TT}:=\{(\tensU_1,\dots,\tensU_d)\!:\tensU_k\in\mathbb{R}^{r_k\times n_k\times r_{k+1}},r_1=r_{d+1}=1,\rank((\tensU_k)_{(1)}) = \rank((\tensU_{k-1})_{(3)}) = r_k\}$, Lie group $\GL(\vecr)$ and group action in a similar fashion as TR by breaking the edge between $\tensU_d$ and $\tensU_1$ in~\cref{fig:gauge_inv}.

Recall that~\cref{tab:TT_TR_quotient} summarizes the quotient geometries of the TT and TR decompositions. There are several essential distinctions: 
\begin{itemize}
    \item Rank conditions. For TT, the rank conditions are imposed on the {mode-1} and the mode-3 unfoldings of core tensors $\tensU_k$ in $\bartensM_{\vecr}^\mathrm{TT}$. For TR, the injectivity is encoded through the full-rank condition on the mode-2 unfolding of each core $\tensU_k$, which is stricter since $\rank((\tensU_k)_{(2)})=r_{k}r_{k+1}$ indicates that $\rank((\tensU_k)_{(1)}) = \rank((\tensU_{k-1})_{(3)}) = r_k$. The full-rank condition for TR is necessary since the group action $\theta$ need not be free even though $\tensU_k$ satisfies the rank condition for TT (e.g., consider $\tensU_k$ defined by $\matu_k(i_k)=\matI_r$ with $r=r_1=r_2=\cdots=r_d$). 
    \item The Lie groups. The Lie group for TT is the product of general linear groups $\GL(\vecr)$. The Lie group for TR is the projective general linear group $\PGL(\vecr)$, since all cores are connected through a single closed loop. 
    \item Quotient manifold construction. The quotient manifold structure for TT follows from the results in~\cite{uschmajew2013geometry,haegeman2014geometry}. However, establishing the quotient manifold structure of TR is significantly more challenging due to the ring structure, where the properness depends on the smallest singular value of mode-2 unfolding of core tensors; see~\cref{thm:TRquotient}.
\end{itemize}

\subsection{Geometry of uniform tensor ring decomposition}
In the same spirit, we provide the geometric tools of the uniform tensor ring decomposition (uTR, or translation-invariant MPS in physics), where all the core tensors are identical, i.e., $\tensU=\tensU_1=\tensU_2=\cdots=\tensU_d$. It is easy to see that the uTR decomposition is not unique due to the gauge invariance, as illustrated in~\cref{fig:gauge_inv_uMPS}. However, the translation-invariant structure further restricts the invertible matrices to be identical across all modes, namely, $\mata=\mata_1=\mata_2=\cdots=\mata_d$ in~\cref{fig:gauge_inv}. 

\begin{figure}[htbp]
    \centering
    \begin{tikzpicture}\scriptsize
        \coordinate (O) at (0,0); 
        \def\l{1.4};


		\node[circle,draw,inner sep=1pt] (U1) at (O) {$\tensU$};
		\node[circle,draw,inner sep=1pt] (U2) at ($(O)+(0.5*\l,-0.866*\l)$) {$\tensU$};
		\node[circle,draw,inner sep=1pt] (U3) at ($(O)+(1.5*\l,-0.866*\l)$) {$\tensU$};
		\node[circle,draw,inner sep=1pt] (U4) at ($(O)+(2*\l,0)$) {$\tensU$};
		\node[circle,draw,inner sep=1pt] (U5) at ($(O)+(1.5*\l,0.866*\l)$) {$\tensU$};
		\node[circle,draw,inner sep=1pt] (U6) at ($(O)+(0.5*\l,0.866*\l)$) {$\tensU$};
        
        \draw[-] (U6)--(U1)--(U2)--(U3)--(U4)-- (U5);
        \draw[dashed] (U5)--(U6);
        \draw (U1) -- ($(U1)+0.5*(-\l,0)$);
        \draw (U2) -- ($(U2)+0.5*(-0.5*\l,-0.866*\l)$);
        \draw (U3) -- ($(U3)+0.5*(0.5*\l,-0.866*\l)$);
        \draw (U4) -- ($(U4)+0.5*(\l,0)$);
        \draw (U5) -- ($(U5)+0.5*(0.5*\l,0.866*\l)$);
        \draw (U6) -- ($(U6)+0.5*(-0.5*\l,0.866*\l)$);

        \node at ($(3*\l,0)$) {\Large$=$};

        \coordinate (O) at ($(4*\l,0)$);
        \node[circle,draw,inner sep=1pt] (U1) at (O) {$\tensU$};
		\node[circle,draw,inner sep=1pt] (U2) at ($(O)+(0.5*\l,-0.866*\l)$) {$\tensU$};
		\node[circle,draw,inner sep=1pt] (U3) at ($(O)+(1.5*\l,-0.866*\l)$) {$\tensU$};
		\node[circle,draw,inner sep=1pt] (U4) at ($(O)+(2*\l,0)$) {$\tensU$};
		\node[circle,draw,inner sep=1pt] (U5) at ($(O)+(1.5*\l,0.866*\l)$) {$\tensU$};
		\node[circle,draw,inner sep=1pt] (U6) at ($(O)+(0.5*\l,0.866*\l)$) {$\tensU$};
        
        \coordinate (A1) at ($0.35*(U6)+0.65*(U1)$);
        \coordinate (A2) at ($0.35*(U1)+0.65*(U2)$);
        \coordinate (A3) at ($0.35*(U2)+0.65*(U3)$);
        \coordinate (A4) at ($0.35*(U3)+0.65*(U4)$);
        \coordinate (A5) at ($0.35*(U4)+0.65*(U5)$);
        \coordinate (A6) at ($0.35*(U5)+0.65*(U6)$);
        
        \coordinate (A1inv) at ($0.35*(U1)+0.65*(U6)$);
        \coordinate (A2inv) at ($0.35*(U2)+0.65*(U1)$);
        \coordinate (A3inv) at ($0.35*(U3)+0.65*(U2)$);
        \coordinate (A4inv) at ($0.35*(U4)+0.65*(U3)$);
        \coordinate (A5inv) at ($0.35*(U5)+0.65*(U4)$);
        \coordinate (A6inv) at ($0.35*(U6)+0.65*(U5)$);
        
        \fill (A1) circle (2pt);
        \fill (A2) circle (2pt);
        \fill (A3) circle (2pt);
        \fill (A4) circle (2pt);
        \fill (A5) circle (2pt);
        \fill (A6) circle (2pt);

        \fill (A1inv) circle (2pt);
        \fill (A2inv) circle (2pt);
        \fill (A3inv) circle (2pt);
        \fill (A4inv) circle (2pt);
        \fill (A5inv) circle (2pt);
        \fill (A6inv) circle (2pt);
        
        
        \draw (U1) -- ($(U1)+0.5*(-\l,0)$);
        \draw (U2) -- ($(U2)+0.5*(-0.5*\l,-0.866*\l)$);
        \draw (U3) -- ($(U3)+0.5*(0.5*\l,-0.866*\l)$);
        \draw (U4) -- ($(U4)+0.5*(\l,0)$);
        \draw (U5) -- ($(U5)+0.5*(0.5*\l,0.866*\l)$);
        \draw (U6) -- ($(U6)+0.5*(-0.5*\l,0.866*\l)$);

        \draw[-] (A6)--(U6)--(U1)--(U2)--(U3)--(U4)-- (U5)--(A6inv);
        \draw[dashed] (A6inv)--(A6);

        \node at ($(A1)+0.2*(-0.866*\l,0.5*\l)$) {$\mata$};
        \node at ($(A1inv)+0.2*(-0.866*\l,0.5*\l)$) {$\mata^{-1}$};
        
        \node at ($(A2)+0.2*(-0.866*\l,-0.5*\l)$) {$\mata$};
        \node at ($(A2inv)+0.2*(-0.866*\l,-0.5*\l)$) {$\mata^{-1}$};
        
        \node[below] at ($(A3)$) {$\mata$};
        \node[below] at ($(A3inv)$) {$\mata^{-1}$};
        
        \node at ($(A4)+0.2*(0.866*\l,-0.5*\l)$) {$\mata$};
        \node at ($(A4inv)+0.2*(0.866*\l,-0.5*\l)$) {$\mata^{-1}$};
        
        \node at ($(A5)+0.2*(0.866*\l,0.5*\l)$) {$\mata$};
        \node at ($(A5inv)+0.2*(0.866*\l,0.5*\l)$) {$\mata^{-1}$};
        
        \node[above] at ($(A6)$) {$\mata$};
        \node[above] at ($(A6inv)$) {$\mata^{-1}$};
        
    \end{tikzpicture}
    \caption{Gauge invariance of uniform tensor ring decomposition.}
    \label{fig:gauge_inv_uMPS}
\end{figure}

We denote the parameter space and quotient set of the injective uniform tensor ring decomposition by
\[\bartensM_r^*:=\mathbb{R}^{r\times n\times r}_*\quad\text{and}\quad\tensM_r^*:=\bartensM_r^*/\PGL(r),\]
where $\PGL(r)=\GL(r)/\mathbb{R}_*$. Note that a single integer $r$ is considered rather than an array. 
\begin{definition}[injective uTR tensor]
    A tensor $\tensX=\llbracket\tensU,\tensU,\dots,\tensU\rrbracket$ with $\tensU\in\bartensM_r^*$ is referred to as an injective uTR tensor.
\end{definition}

In fact, $\tensM_r^*$ is also a quotient manifold. A proof can be found in~\cite[Theorem 20]{haegeman2014geometry}. By using~\cref{prop:vertical_and_horizontal}, we provide the vertical and horizontal spaces and projections. It is worth noting that the computational cost of the projections in~\cref{eq:v_h_uTR} is independent of the order $d$, since the core tensors in uTR are identical, whereas the cost of the TR-related projections in~\cref{eq:v_h_TR} scales with $d$.


\begin{proposition}
	{{\color{black}The quotient set $\tensM_r^*$ is a quotient manifold with dimension 
    \[\dim(\tensM_r^*)=nr^2-r^2+1.\]
	Given $\tensU\in\bartensM_r^*$, the vertical and horizontal spaces are given by}}
	\begin{equation}
		\begin{aligned}
			\vertical_\tensU\!\bartensM_r^*&=\{\tensU\times_1\matD-\tensU\times_3\matD^\top,\ \tr(\matD)=0\},\\
			\horizontal_\tensU\!\bartensM_r^*&=\{\xi\in\mathbb{R}^{r\times n\times r}:\xi_{(1)}\matu_{(1)}^\top=\matu_{(3)}\xi_{(3)}^{\top}\}.
		\end{aligned}
	\end{equation}
	Furthermore, the orthogonal projection of a tensor $\tensV\in\mathbb{R}^{r\times n\times r}$ onto the vertical and horizontal spaces are given by
	\begin{equation}\label{eq:v_h_uTR}
		\begin{aligned}
			\proj_{\vertical_\tensU\!\bartensM_r^*}(\tensV)&=\tensU\times_1\matD-\tensU\times_3\matD^\top\\
			\proj_{\horizontal_\tensU\!\bartensM_r^*}(\tensV)&=\tensV-(\tensU\times_1\matD-\tensU\times_3\matD^\top),
		\end{aligned}
	\end{equation}
	where $\matD\in\mathbb{R}^{r\times r}$ is determined by the matrix equations
	\begin{equation}\label{eq:proj_uTR}
		\begin{aligned}
			(\matv_{(1)}-\matd\matu_{(1)}+\matu_{(1)}(\matd\otimes\matI_{n}))\matu_{(1)}^\top&=\matu_{(3)}(\matv_{(3)}-\matu_{(3)}(\matI_{n}\otimes\matd^\top)+\matd^\top\matu_{(3)})^\top,\\
			\tr(\matd)&=0.
		\end{aligned}
	\end{equation}
\end{proposition}

{{\color{black}
\paragraph{Computational complexity}
Denote $n:=\max_{k\in[d]}n_k$, and $r:=\max_{k\in[d]}r_k$. The matrix equations in~\cref{eq:proj_uTR} contain $\mathcal{O}(r^2)$ undetermined variables. Following a similar analysis in~\cref{eq: projection}, we obtain that the coefficient matrix can be formed in $\mathcal{O}(nr^4)$ flops. Solving the linear system requires $\mathcal{O}(r^6)$ flops, and constructing the vertical and horizontal projections requires $\mathcal{O}(nr^3)$ flops. Therefore, the overall computational complexity is $\mathcal{O}(nr^4+r^6)$, which is independent of the tensor order $d$.
}}

\section{Geometric methods on the quotient manifolds}\label{sec:geometric_methods}
Based on the developed geometries, we provide geometric methods for minimizing a smooth function on the quotient manifolds.

\subsection{Riemannian optimization on tensor ring}
We consider the following optimization on tensors in the tensor ring format
\begin{equation}
	\label{eq:opt_on_Mr}
	\min_{\vectensU\in\bartensM_\vecr^*}\ \barf(\vectensU)=\phi(\tau(\vectensU)),
\end{equation}
where $\phi:\mathbb{R}^{n_1\times n_2\times\cdots\times n_d}\to\mathbb{R}$ is a smooth function, $\tau(\vectensU)=\llbracket\tensU_1,\tensU_2,\dots,\tensU_d\rrbracket$ recovers a tensor via tensor ring decomposition from parameters $\tensU_1,\tensU_2,\dots,\tensU_d$. We observe that $\barf$ admits the invariance $\barf(\vectensU)=\barf(\vectensV)$ for all $\vectensU\sim\vectensV$. Therefore, $f:=\barf\circ\pi^{-1}$ is well-defined on the quotient manifold $\tensM_\vecr^*=\bartensM_\vecr^*/\PGL(\vecr)$ and one can consider solving~\cref{eq:opt_on_Mr} by Riemannian optimization methods on $\tensM_\vecr^*$; see~\cref{fig: comm diagram func}.
\begin{figure}[htbp]
	\centering
	\begin{tikzpicture}
		\def\x{1.6};

		\node (Mbar) at (0,\x) {$\bartensM_\vecr^*$};
		\node (M) at (0,0) {$\tensM_\vecr^*$};
		\node (N) at (2*\x,0) {$\mathbb{R}$};
		
		\draw[->] (Mbar) -- (M);
		\draw[->] (M) -- (N);
		\draw[->] (Mbar) -- (N);
		
		\node[left] at ($0.5*(Mbar)+0.5*(M)$) {$\pi$};
		\node[below] at ($0.5*(M)+0.5*(N)$) {$f=\barf\circ\pi^{-1}$};
		\node[above] at ($0.5*(Mbar)+0.5*(N)$) {$\barf$};
	\end{tikzpicture}
	\caption{Commutative diagram of smooth functions on the total space $\bartensM_\vecr^*$ and quotient manifold $\tensM_\vecr^*$.}
	\label{fig: comm diagram func}
\end{figure}

Specifically, the Riemannian gradient descent (TR-RGD(Q)) method updates a point  $[\vectensU^{(t)}]\in\bartensM_\vecr^*/\PGL(\vecr)$ by 
\[[\vectensU^{(t+1)}]=\retr_{[\vectensU^{(t)}]}(-s^{(t)}\grad f([\vectensU^{(t)}])).\]
Recall that the retraction $\retr$ on $\bartensM_\vecr^*/\PGL(\vecr)$ is defined by $\retr_{[\vectensU]}(\xi)=[\bar{\retr}_\vectensU(\lift_\vectensU(\xi))]$ for $\xi\in\tangent_{[\vectensU]}\!\tensM_\vecr^*$ and the Riemannian gradient of $f$ at $[\vectensU]$ satisfies $\lift_\vectensU(\grad f([\vectensU]))=\grad\barf(\vectensU)$ and 
$\grad\barf(\vectensU)\in\tangent_\vectensU\!\bartensM_\vecr^*$ such that $\mathrm{D}\barf(\vectensU)[\vec\eta]=\langle\grad\barf(\vectensU),\vec\eta\rangle$ for all $\vec\eta\in\tangent_\vectensU\!\bartensM_\vecr^*$, where $s^{(t)}>0$ and $\bar{\retr}$ is a retraction on $\bartensM_\vecr^*$. Therefore, the update rule of TR-RGD(Q) is equivalent to
\[[\vectensU^{(t+1)}]=\retr_{[\vectensU^{(t)}]}(-s^{(t)}\grad f([\vectensU^{(t)}]))=[\bar{\retr}_{\vectensU^{(t)}}(-s^{(t)}\grad \barf(\vectensU^{(t)}))],\]
which is implemented on the total space $\bartensM_\vecr^*$. Moreover, for minimization of $\barf$ on $\bartensM_\vecr^*$, the Riemannian gradient descent method (TR-RGD(E)) updates a point $\vectensU^{(t)}\in\bartensM_\vecr^*$ by $\vectensU^{(t+1)}=\bar{\retr}_{\vectensU^{(t)}}(-s^{(t)}\grad \barf(\vectensU^{(t)}))$. Therefore, the RGD method to minimize $\barf$ on $\bartensM_\vecr^*$ is numerically equivalent to the RGD method to minimize $f$ on the quotient manifold $\bartensM_\vecr^*/\PGL(\vecr)$.  We refer to~\cite[\S 9]{boumal2023intromanifolds} for more details. The TR-RGD(Q) method is summarized in Algorithm~\ref{alg: RGD(Q)}.

The Riemannian conjugate gradient (TR-RCG(Q)) method is given in Algorithm~\ref{alg: RCG(Q)}. The vector transport is chosen as the projection onto the horizontal space $\proj_{\horizontal_{\vectensU^{(t)}}\!\bartensM_\vecr^*}$ and the retraction is chosen as the identity map. Since $\bartensM_\vecr^*$ is not compact, the sequence generated by TR-RGD(Q) can be unbounded. One may adopt a regularization term $\lambda\|\vectensU\|_\frob^2/2$ with $\lambda>0$ to ensure the coercivity of the cost function (see, e.g., \cite{gao2024riemannian}) and thus TR-RGD(Q) and TR-RCG(Q) methods enjoy the global convergence. 

\begin{algorithm}[htbp]
	\caption{Riemannian gradient {{\color{black}descent}} method on $\tensM_\vecr^*$ (TR-RGD(Q))}\label{alg: RGD(Q)}
	\begin{algorithmic}[1]
		\REQUIRE Manifold $\bartensM_\vecr^*$, smooth function $\barf$, initial guess $\vectensU^{(0)}\in\bartensM_\vecr^*$, $t=0$. 
		\WHILE{the stopping criteria are not satisfied}
		\STATE Compute $\eta^{(t)}=-\grad\barf(\vectensU^{(t)})$.
		\STATE Select a stepsize $s^{(t)}$.
		\STATE Update $\vectensU^{(t+1)}=\bar{\retr}_{\vectensU^{(t)}}(s^{(t)}\eta^{(t)})$; $t=t+1$.
		\ENDWHILE
		\ENSURE $\vectensU^{(t)}\in\bartensM_\vecr^*$. 
	\end{algorithmic}
\end{algorithm}

\begin{algorithm}[htbp]
	\caption{Riemannian conjugate gradient method on $\tensM_\vecr^*$ (TR-RCG(Q))}\label{alg: RCG(Q)}
	\begin{algorithmic}[1]
		\REQUIRE Manifold $\bartensM_\vecr^*$, smooth function $\barf$, initial guess $\vectensU^{(0)}\in\bartensM_\vecr^*$, $t=0$  $\beta^{(0)}=0$. 
		\WHILE{the stopping criteria are not satisfied}
		\STATE Compute $\eta^{(t)}=- \grad\barf(\vectensU^{(t)})+\beta^{(t)}\proj_{\horizontal_{\vectensU^{(t)}}}\eta^{(t-1)}$ with CG parameter $\beta^{(t)}$.
		\STATE Select a stepsize $s^{(t)}$.
		\STATE Update $\vectensU^{(t+1)}=\bar{\retr}_{\vectensU^{(t)}}(s^{(t)}\eta^{(t)})$; $t=t+1$.
		\ENDWHILE
		\ENSURE $\vectensU^{(t)}\in\bartensM_\vecr^*$. 
	\end{algorithmic}
\end{algorithm}




\subsection{Riemannian optimization on uniform tensor ring}
Similarly, we consider the following optimization on tensors in the uniform tensor ring format
\begin{equation}
	\label{eq:opt_on_uTR}
	\min_{\tensU\in\bartensM_r^*}\ \barf(\tensU)=\phi(\tau(\tensU)),
\end{equation}
where $\phi:\mathbb{R}^{n_1\times n_2\times\cdots\times n_d}\to\mathbb{R}$ is a smooth function, $\tau(\tensU)=\llbracket\tensU,\tensU,\dots,\tensU\rrbracket$ recovers a tensor via the uniform tensor ring decomposition from $\tensU$. Since $\barf$ admits the invariance $\barf(\tensU)=\barf(\tensV)$ for all $\tensU\sim\tensV$, $f:=\barf\circ\pi^{-1}$ is also well-defined on the quotient manifold $\tensM_r^*=\bartensM_r^*/\PGL(r)$ and one can consider solving~\cref{eq:opt_on_uTR} by Riemannian optimization methods on $\tensM_r^*$. The Riemannian gradient descent (uTR-RGD(Q)) and the Riemannian conjugate gradient (uTR-RCG(Q)) methods are listed in~\cref{alg: uTR-RGD(Q),alg: uTR-RCG(Q)}.

\begin{algorithm}[htbp]
	\caption{Riemannian gradient {{\color{black}descent}} method on $\bartensM_r^*$ (uTR-RGD(Q))}\label{alg: uTR-RGD(Q)}
	\begin{algorithmic}[1]
		\REQUIRE Manifold $\bartensM_r^*$, smooth function $\barf$, initial guess $\tensU^{(0)}\in\bartensM_r^*$, $t=0$. 
		\WHILE{the stopping criteria are not satisfied}
		\STATE Compute $\eta^{(t)}=-\grad\barf(\tensU^{(t)})$.
		\STATE Select a stepsize $s^{(t)}$.
		\STATE Update $\tensU^{(t+1)}=\bar{\retr}_{\tensU^{(t)}}(s^{(t)}\eta^{(t)})$; $t=t+1$.
		\ENDWHILE
		\ENSURE $\tensU^{(t)}\in\bartensM_r^*$. 
	\end{algorithmic}
\end{algorithm}
\begin{algorithm}[htbp]
	\caption{Riemannian conjugate gradient method on $\bartensM_r^*$ (uTR-RCG(Q))}\label{alg: uTR-RCG(Q)}
	\begin{algorithmic}[1]
		\REQUIRE Manifold $\bartensM_r^*$, smooth function $\barf$, initial guess $\vectensU^{(0)}\in\bartensM_r^*$, $t=0$  $\beta^{(0)}=0$. 
		\WHILE{the stopping criteria are not satisfied}
		\STATE Compute $\eta^{(t)}=- \grad\barf(\tensU^{(t)})+\beta^{(t)}\proj_{\horizontal_{\tensU^{(t)}}}\eta^{(t-1)}$ with CG parameter $\beta^{(t)}$.
		\STATE Select a stepsize $s^{(t)}$.
		\STATE Update $\tensU^{(t+1)}=\bar{\retr}_{\tensU^{(t)}}(s^{(t)}\eta^{(t)})$; $t=t+1$.
		\ENDWHILE
		\ENSURE $\vectensU^{(t)}\in\bartensM_r^*$. 
	\end{algorithmic}
\end{algorithm}

\subsection{Discussion: different Riemannian metric}
Since different metrics result in different Riemannian gradients and thus distinct Riemannian methods, one may consider choosing an appropriate Riemannian metric $\barg$ to accelerate the Riemannian methods~\cite{gao2025optimization}. For instance, a \emph{preconditioned} metric on the parameter space $\bartensM_\vecr$ of tensor ring decomposition can be proposed~\cite{gao2024riemannian} by incorporating second-order information of the objective function of low-rank tensor ring completion,
\[\barg_{\vectensU}(\vec\xi,\vec\eta)=\sum_{k=1}^d\langle(\xi_k)_{(2)},(\eta_k)_{(2)}(\matW_{\neq k}^\top\matW_{\neq k}^{}+\delta\matI_{r_kr_{k+1}})\rangle\]
for $\vec\xi,\vec\eta\in\tangent_\vectensU\!\bartensM_\vecr\simeq\bartensM_\vecr$, where $\delta>0$ is introduced to ensure the positive definiteness. It is worth noting that for $[\vectensU]\in\tensM_\vecr^*$, the matrix $\matW_{\neq k}^\top\matW_{\neq k}^{}$ is naturally positive definite from~\cref{lem: rank} and thus the regularization term $\delta\matI_{r_k r_{k+1}}$ is not {{\color{black}necessary}}.

{{\color{black}In general, the evaluations of the cost function and the Euclidean gradient are the dominant computational steps in Riemannian optimization algorithms. Among the geometric operations, the projection $\proj_{\horizontal_{\vectensU}\!\bartensM_\vecr^*}$ dominates the computational cost for large tensor order $d$ and tensor rank $\vecr$ since solving the linear system of $\mathcal{O}(dr^2)$ variables is involved.}}

\section{Numerical validation}\label{sec: numerical}
In this section, we numerically validate the geometric tools via the (uniform) tensor ring completion tasks. Given a partially observed $d$-th order tensor $\tensA\in\mathbb{R}^{n_1\times\cdots\times n_d}$ on an index set $\Omega\subseteq[n_1]\times\cdots\times[n_d]$, tensor completion aims to recover the full tensor $\tensA$ from entries on $\Omega$. We formulate the tensor completion task on the parameter space $\tensM_\vecr^*$ of injective TR tensors with rank $\vecr=(r_1,r_2,\dots,r_d)$, i.e.,
\begin{equation}\label{eq:LRTCTR}
	\begin{aligned}
		\min\ &\ \barf(\vectensU)=\frac{1}{2}\left\| \proj_\Omega(\llbracket\tensU_1,\tensU_2,\dots, \tensU_d\rrbracket)-\proj_\Omega(\tensA)\right\|_\mathrm{F}^2\\ 
		\subjectto\ &\ \vectensU=(\tensU_1,\tensU_2,\dots,\tensU_d)\in\bartensM_\vecr^*,
	\end{aligned}
\end{equation}
where $\proj_\Omega$ denotes the projection operator onto $\Omega$ defined by $\proj_\Omega(\tensX)(i_1,\dots,i_d)=\tensX(i_1,\dots,i_d)$ if $(i_1,\dots,i_d)\in\Omega$, otherwise $\proj_\Omega(\tensX)(i_1,\dots,i_d)=0$. Since $\barf(\vectensU)=\barf(\vectensV)$ holds for all $\vectensU,\vectensV\in\bartensM_\vecr^*$ with $\vectensU\sim\vectensV$, the problem~\cref{eq:LRTCTR} can also be formulated on the quotient manifold $\tensM_\vecr^* = \bartensM_\vecr^*/\PGL(\vecr)$. Therefore, we can adopt~\cref{alg: RGD(Q),alg: RCG(Q)} to solve problem~\cref{eq:LRTCTR}. Similarly, the uniform tensor ring completion task can be formulated by
\begin{equation}\label{eq:LRTCuMPS}
	\min_{\tensU\in\bartensM_r^*}\ \ \barf(\tensU)=\frac{1}{2}\left\| \proj_\Omega(\llbracket\tensU,\tensU,\dots, \tensU\rrbracket)-\proj_\Omega(\tensA)\right\|_\mathrm{F}^2, 
\end{equation}
and quotient manifold-based methods (\cref{alg: uTR-RGD(Q),alg: uTR-RCG(Q)}) are also available for solving~\cref{eq:LRTCuMPS}.

All proposed geometric methods in~\cref{alg: RGD(Q)}--\cref{alg: uTR-RCG(Q)} are implemented in Manopt v7.1.0~\cite{boumal2014manopt}, a Matlab library for Riemannian methods. The implementation of tensor ring computation is based on the LRTCTR toolbox\footnote{Available at \url{https://github.com/JimmyPeng1998/LRTCTR}.}. All experiments are performed on a workstation with two Intel(R) Xeon(R) Processors Gold 6330 (at 2.00GHz$\times$28, 42M Cache) and 512GB of RAM running Matlab R2019b under Ubuntu 22.04.3. The code for proposed methods is available at~\url{https://github.com/JimmyPeng1998/LRTCTR}.

\subsection{Test on different geometries}
We compare the numerical performance of RGD and RCG methods on the total and quotient manifolds, i.e., optimization on $\bartensM_\vecr^*$ and $\tensM_\vecr^*$ (or $\bartensM_r^*$ and $\tensM_r^*$), respectively. The RGD and RCG methods for minimization of $\barf$ on the total space $\bartensM_\vecr^*$ (or $\bartensM_r^*$) are denoted by TR-RGD(E) (uTR-RGD(E)) and TR-RCG(E) (uTR-RCG(E)), respectively. We generate a third-order synthetic tensor $\tensX^*$ with rank $\vecr = (2, 2, 2)$ and $n_1=n_2=n_3=100$, where elements of the core tensors are i.i.d. samples from the uniform distribution on $[0,1]$. \cref{fig:recovery} reports the recovery performance of TR-based and uniform TR-based methods. We observe that the RCG methods {{\color{black}exhibit different convergence behavior}}, since the vector transport on the quotient manifold is different from the total manifold. The TR-RCG(Q) method performs the best among all candidates in~\cref{fig:recovery} (above). Additionally, RGD methods for both geometries have the same numerical performance {{\color{black}in terms of both the iteration and computation time}}, which coincides with the theoretical result. 
\begin{figure}[htbp]
	\centering
	\includegraphics[width=\textwidth]{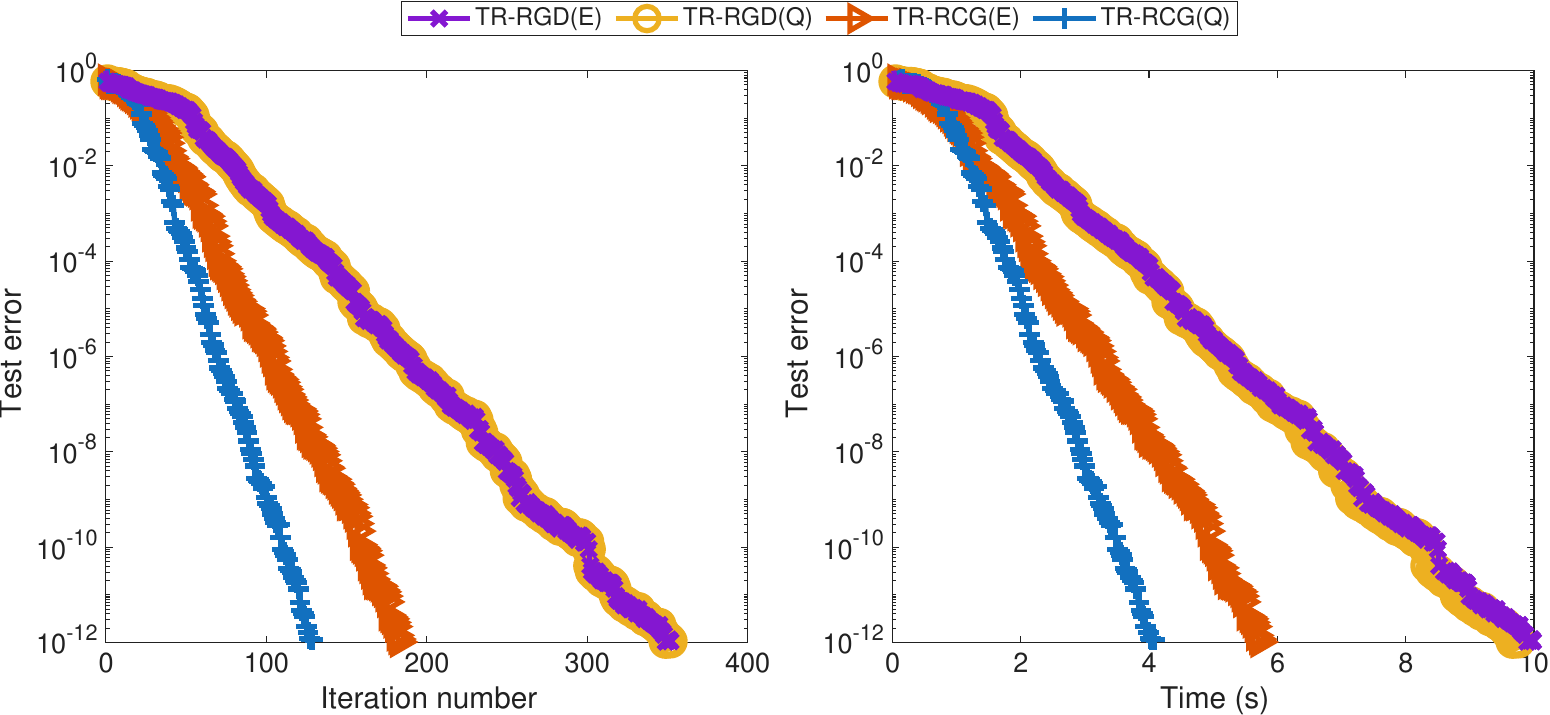}

    \includegraphics[width=\textwidth]{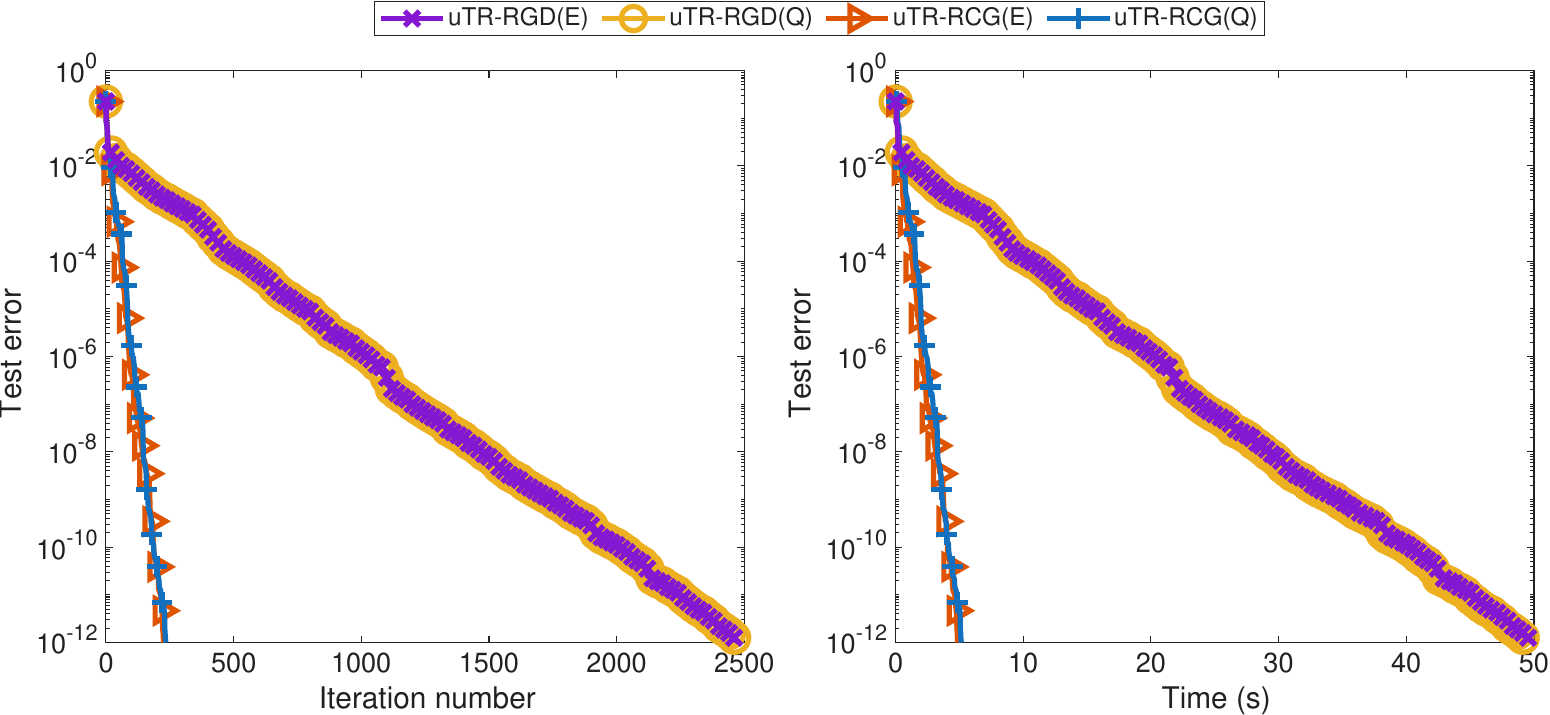}
	\caption{{{\color{black}Recovery performance of TR-based (above) and uniform TR-based methods (below) in tensor ring completion under different geometries.}}}
	\label{fig:recovery}
\end{figure}

{{\color{black}
\subsection{Test on higher-order tensors}
In order to show the scalability of the proposed quotient-based methods, we compare the performance of TR-based methods on higher-order tensors. We consider the recovery of a larger tensor order with $d=5,n=100,\vecr=(4,4,4,4,4)$, and report the average time per iteration for different tensor order $d=3,4,5,6,7$ and tensor rank $\vecr=(r,r,\dots,r)$ with $r=2,3,4,5,6$; see~\cref{fig:scaling}. The size of training set is selected to be $|\Omega|=10ndr^2$. We observe that the TR-RCG(Q) method can reduce the number of iterations while maintaining a comparable computational cost for $d\leq 7$ and $r\leq 6$.

\begin{figure}[htbp]
    \centering
    \includegraphics[width=\textwidth]{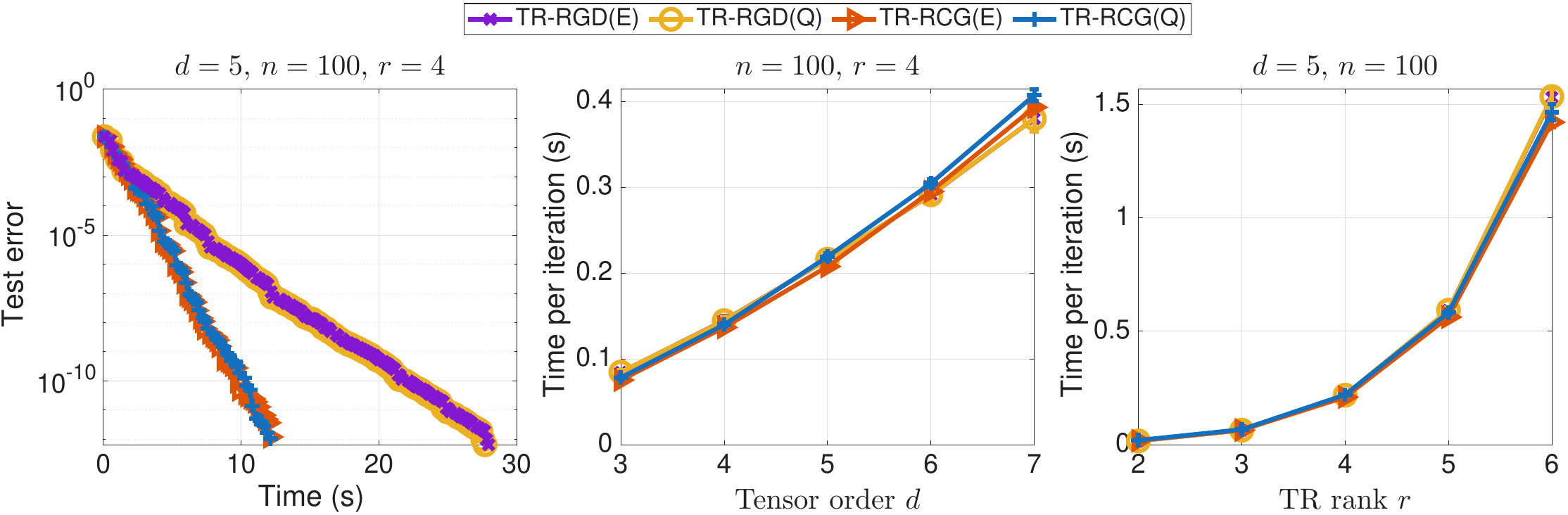}
    \caption{Numerical results on synthetic datasets under the total and quotient spaces. Left: test errors of quotient-based methods for $d=5$. Middle: average time per iteration for different tensor order $d=3,4,5,6,7$. Right: average time per iteration for different tensor rank $\vecr=(r,r,\dots,r)$ with $r=2,3,4,5,6$.}
    \label{fig:scaling}
\end{figure}
}}

\subsection{Test on sample complexity}
We numerically investigate the relationship between sampling rate and tensor size $n$ for tensor ring completion. We consider two scenarios: 1) for TR decomposition, we randomly generate third-order synthetic tensors $\tensX^*$ with rank $\vecr = (2, 2, 2)$, $n_1=n_2=n_3\in\{50,60,\dots,200\}$ and $|\Omega|\in\{2000,4000,\dots,100000\}$; 2) for uniform TR decomposition, we set $r=2$, $n_1=n_2=n_3\in\{50,100,\dots,200\}$ and $|\Omega|\in\{1000,2000,\dots,20000\}$. For each combination of $(n,|\Omega|)$, we run the TR-RCG(Q) and uTR-RCG(Q) five times. Figure~\ref{fig: phase completion} illustrates the phase transition behavior of TR-RCG(Q) and uTR-RCG(Q) with respect to the tensor size $n$ and sample size $|\Omega|$. For both methods, we observe a transition from unsuccessful (black) to successful (white) recovery as $|\Omega|$ increases. In contrast with the TR-RCG(Q) method, the uTR-RCG(Q) method requires substantially fewer samples for successful recovery since the uniform TR decomposition significantly reduces the number of parameters.
\begin{figure}[htbp]
	\centering
	{\includegraphics[width=0.48\textwidth]{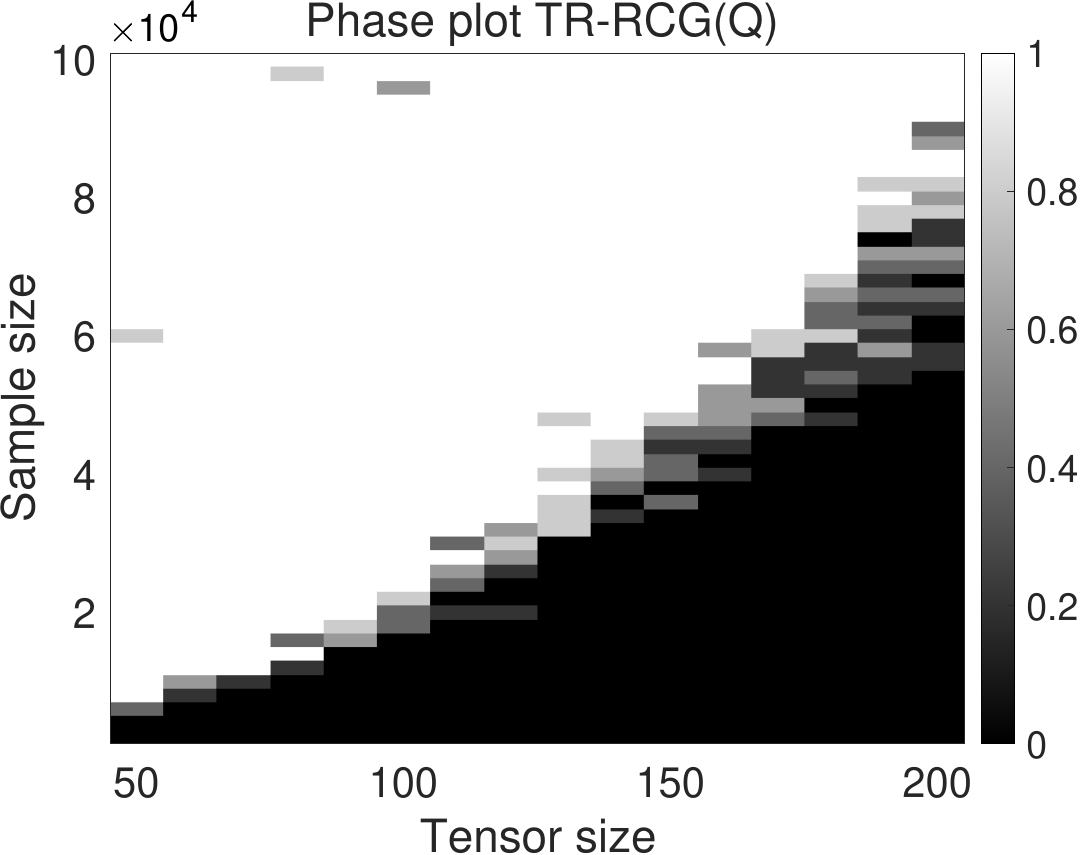}}
	{\includegraphics[width=0.48\textwidth]{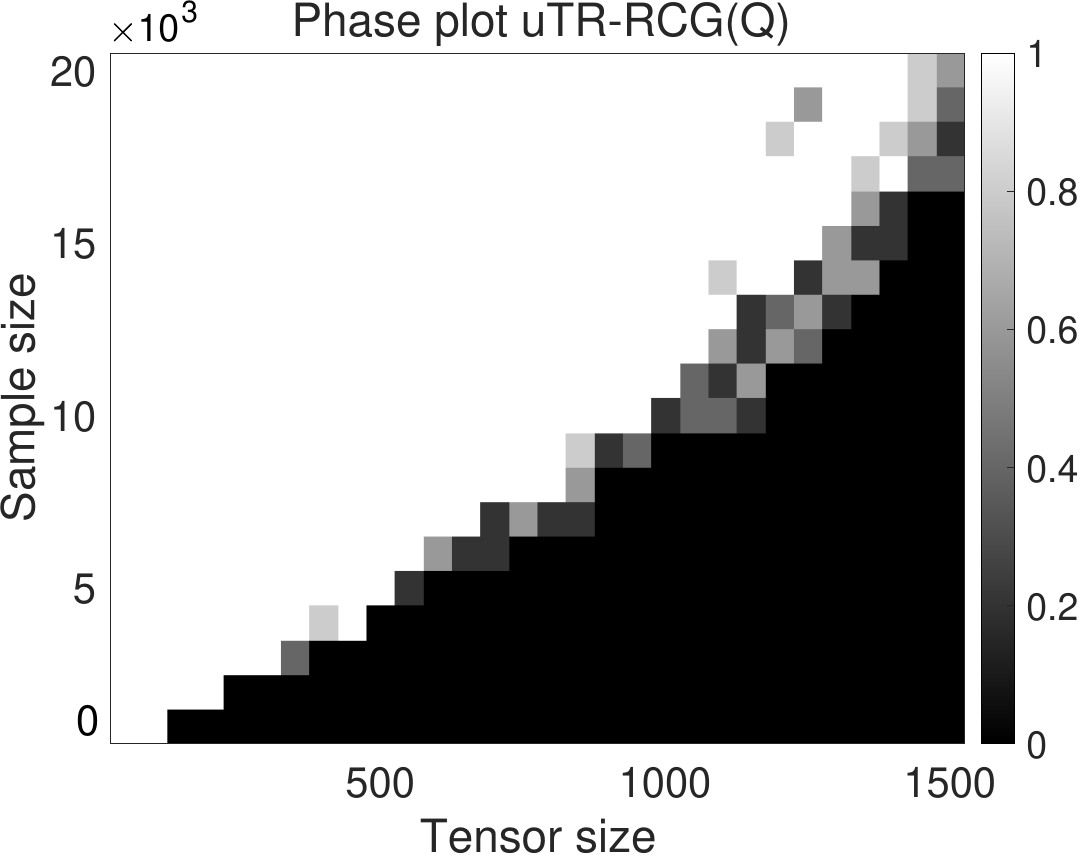}}
	\caption{Phase plots of the TR-RCG(Q) and uTR-RCG(Q) methods of low-rank tensor completion. Left: TR-RCG(Q); right: uTR-RCG(Q).}
	\label{fig: phase completion}
\end{figure}

{{\color{black}
\subsection{Test on real-world dataset}
We consider tensor completion on a real-world movie rating dataset {``MovieLens 1M"\footnote{Available at \url{https://grouplens.org/datasets/movielens/1m/}.}}, which consists of one million movie ratings from 6040 users on 3952 movies from September, 1997 to April, 1998. By choosing one week as a period, we formulate the movie ratings as a third-order tensor $\tensA$ of size $6040\times 3952\times 150$. We randomly select $80\%$ of the known ratings as a training set $\Omega$ and the rest $20\%$ ratings are test set $\Gamma$. The rank parameter is set to be $\vecr=(6,10,3)$ by following~\cite[\S 5.2]{gao2024riemannian}. We compare the performance of the proposed methods to the Riemannian methods under preconditioned metrics in~\cite[Algorithms 1 and 2]{gao2024riemannian}, denoted by ``TR-RGD(precon)'' and ``TR-RCG(precon)''. A method is terminated if it exceeds the time budget of $3000$s. 

The numerical results are reported in~\cref{fig:movielens}. We observe that the TR-RCG(Q) method achieves the best performance in terms of test error, i.e., the quotient geometry can improve the performance of Riemannian methods. We also report the singular values of the core tensors in~\cref{fig:movielens} (right), which indicate that the core tensors are well-conditioned and thus the full-rank assumption is reasonable in practice.
\begin{figure}[htbp]
    \centering
    \includegraphics[width=\textwidth]{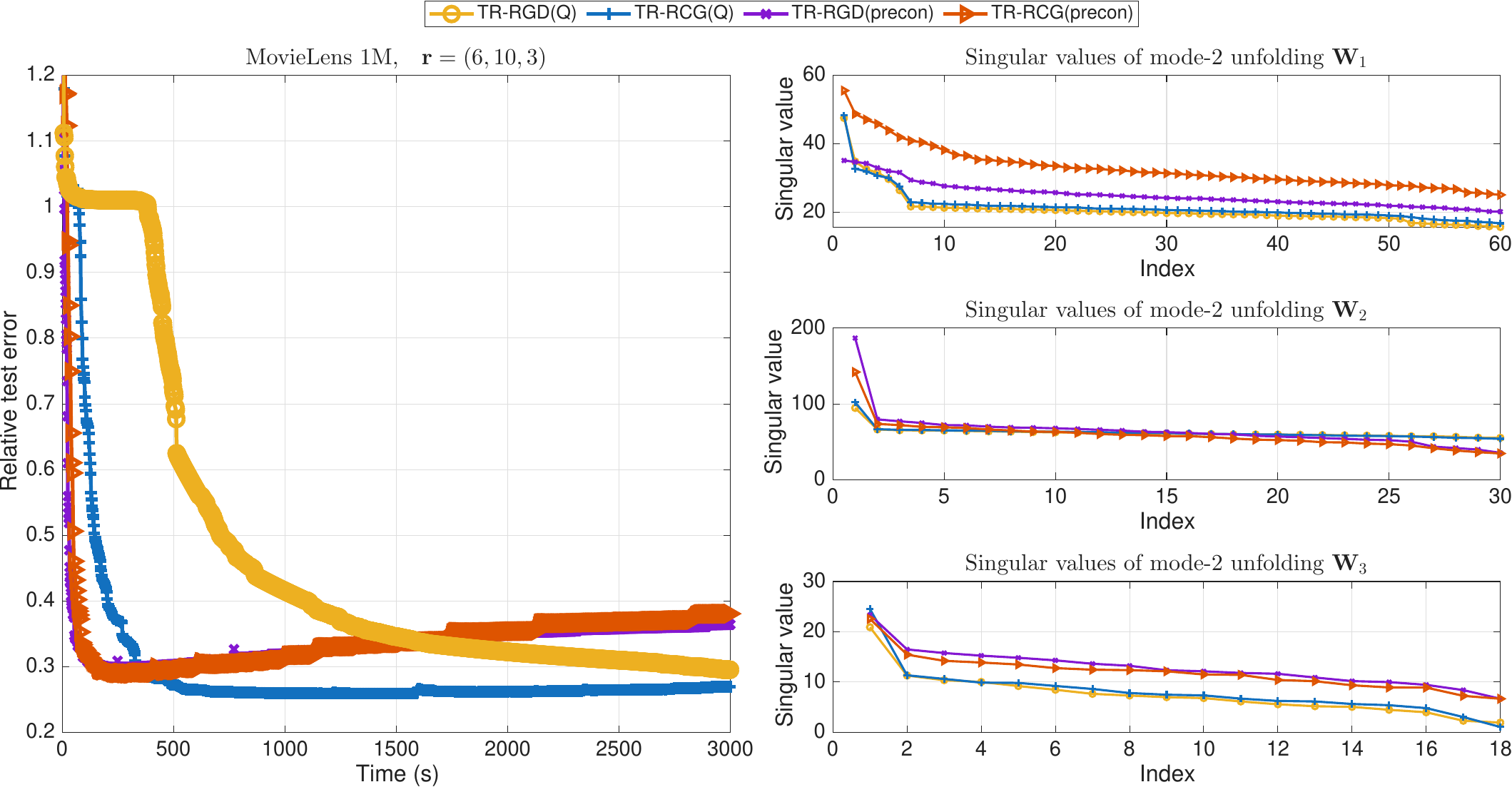}
    \caption{Numerical results on MovieLens 1M dataset. Left: test errors. Right: singular values of core tensors.}
    \label{fig:movielens}
\end{figure}
}}

\section{Conclusion}\label{sec:conclusion}
We have introduced and studied the quotient geometries of the tensor ring decompositions. Unlike the existing results in the geometry of tensors in tensor train format, the geometry of tensor ring decomposition is much more intricate due to the ring structure. To address this challenge, we have imposed full-rank conditions on the parameter space and introduced an appropriate equivalence relation on the new parameter space via gauge invariance. We have proved that the resulting quotient set is a quotient manifold, where the main analytical difficulty arises from the ring structure. {{\color{black}We have also proved that the image of the quotient manifold under the TR reconstruction map is a smooth immersed submanifold of the ambient tensor space.}} Parametrizations of the vertical and horizontal spaces and projections have been provided. We have also extended the results to the uniform TR decomposition. Based on the developed geometries, we have adopted geometric methods for the minimization of a smooth function on the quotient manifolds, and numerically validated the geometric tools in low-rank tensor ring completion.

{{\color{black}In general, we do not claim that the quotient-based methods uniformly perform better than the methods on the total space (under Euclidean or preconditioned metrics) for every rank and tensor order. The quotient geometry provides an intrinsic formulation of the optimization problem and removes directions in the vertical space, which can lead to improved RCG performance in some instances.}}

There are several potential directions for future work. In practice, extending the geometric tools to broader applications remains an interesting direction. For instance, since tensor ring decomposition produces excellent approximations to ground states of local Hamiltonians with modest rank, applying the proposed quotient geometry to efficiently compute multiple eigenvectors or eigenspaces remains an open question. Moreover, designing problem-adapted Riemannian metrics or preconditioners for optimization on TR manifolds may significantly enhance scalability in applications.


 \section*{Acknowledgments}
 The authors would like to thank {{\color{black}the editor and reviewers for the insightful comments to improve the quality of the manuscript, especially encouraging us to explore the immersed submanifold structure of TR decomposition. We also would like to thank}} Haokai Zhang for discussions on matrix product states in physics, and Yingzhou Li for discussions on sample complexity of tensor ring completion.

\bibliographystyle{siamplain}
\bibliography{ref}

\begin{thebibliography}{10}

\bibitem{absil2012projection}
{\sc P.-A. Absil and J.~Malick}, {\em Projection-like retractions on matrix manifolds}, SIAM Journal on Optimization, 22 (2012), pp.~135--158, \url{https://doi.org/10.1137/100802529}.

\bibitem{bachmayr2023low}
{\sc M.~Bachmayr}, {\em Low-rank tensor methods for partial differential equations}, Acta Numerica, 32 (2023), pp.~1--121, \url{https://doi.org/10.1017/S0962492922000125}.

\bibitem{boumal2023intromanifolds}
{\sc N.~Boumal}, {\em An introduction to optimization on smooth manifolds}, Cambridge University Press, 2023, \url{https://doi.org/10.1017/9781009166164}.

\bibitem{boumal2014manopt}
{\sc N.~Boumal, B.~Mishra, P.-A. Absil, and R.~Sepulchre}, {\em Manopt, a {Matlab} toolbox for optimization on manifolds}, The Journal of Machine Learning Research, 15 (2014), pp.~1455--1459.

\bibitem{cai2026tensor}
{\sc J.-F. Cai, W.~Huang, H.~Wang, and K.~Wei}, {\em Tensor completion via tensor train based low-rank quotient geometry under a preconditioned metric}, Numerical Linear Algebra with Applications, 33 (2026), p.~e70056, \url{https://doi.org/10.1002/nla.70056}.

\bibitem{chen2025provably}
{\sc H.~Chen, S.~Chen, and A.~R. Zhang}, {\em A provably efficient method for tensor ring decomposition and its applications}, arXiv preprint arXiv:2512.01016,  (2025).

\bibitem{chen2020tensor}
{\sc Z.~Chen, Y.~Li, and J.~Lu}, {\em Tensor ring decomposition: optimization landscape and one-loop convergence of alternating least squares}, SIAM Journal on Matrix Analysis and Applications, 41 (2020), pp.~1416--1442, \url{https://doi.org/10.1137/19M1270689}.

\bibitem{de2000multilinear}
{\sc L.~De~Lathauwer, B.~De~Moor, and J.~Vandewalle}, {\em A multilinear singular value decomposition}, SIAM journal on Matrix Analysis and Applications, 21 (2000), pp.~1253--1278.

\bibitem{de2008tensor}
{\sc V.~De~Silva and L.-H. Lim}, {\em Tensor rank and the ill-posedness of the best low-rank approximation problem}, SIAM Journal on Matrix Analysis and Applications, 30 (2008), pp.~1084--1127, \url{https://doi.org/10.1137/06066518X}.

\bibitem{dolgov2014computation}
{\sc S.~V. Dolgov, B.~N. Khoromskij, I.~V. Oseledets, and D.~V. Savostyanov}, {\em Computation of extreme eigenvalues in higher dimensions using block tensor train format}, Computer Physics Communications, 185 (2014), pp.~1207--1216, \url{https://doi.org/10.1016/j.cpc.2013.12.017}.

\bibitem{dong2022new}
{\sc S.~Dong, B.~Gao, Y.~Guan, and F.~Glineur}, {\em New {Riemannian} preconditioned algorithms for tensor completion via polyadic decomposition}, SIAM Journal on Matrix Analysis and Applications, 43 (2022), pp.~840--866, \url{https://doi.org/10.1137/21M1394734}.

\bibitem{gao2024desingularization}
{\sc B.~Gao, R.~Peng, and Y.-x. Yuan}, {\em Desingularization of bounded-rank tensor sets}, arXiv preprint arXiv:2411.14093,  (2024).

\bibitem{gao2024riemannian}
{\sc B.~Gao, R.~Peng, and Y.-x. Yuan}, {\em Riemannian preconditioned algorithms for tensor completion via tensor ring decomposition}, Computational Optimization and Applications, 88 (2024), pp.~443--468, \url{https://doi.org/10.1007/s10589-024-00559-7}.

\bibitem{gao2025low}
{\sc B.~Gao, R.~Peng, and Y.-x. Yuan}, {\em Low-rank optimization on {Tucker} tensor varieties}, Mathematical Programming, 214 (2025), pp.~357--407, \url{https://doi.org/10.1007/s10107-024-02186-w}.

\bibitem{gao2025optimization}
{\sc B.~Gao, R.~Peng, and Y.-x. Yuan}, {\em Optimization on product manifolds under a preconditioned metric}, SIAM Journal on Matrix Analysis and Applications, 46 (2025), pp.~1816--1845, \url{https://doi.org/10.1137/24M1643773}.

\bibitem{grasedyck2013literature}
{\sc L.~Grasedyck, D.~Kressner, and C.~Tobler}, {\em A literature survey of low-rank tensor approximation techniques}, GAMM-Mitteilungen, 36 (2013), pp.~53--78, \url{https://doi.org/10.1002/gamm.201310004}.

\bibitem{haegeman2014geometry}
{\sc J.~Haegeman, M.~Mari{\"e}n, T.~J. Osborne, and F.~Verstraete}, {\em Geometry of matrix product states: Metric, parallel transport, and curvature}, Journal of Mathematical Physics, 55 (2014), \url{https://doi.org/10.1063/1.4862851}.

\bibitem{hitchcock1928multiple}
{\sc F.~L. Hitchcock}, {\em Multiple invariants and generalized rank of a p-way matrix or tensor}, Journal of Mathematics and Physics, 7 (1928), pp.~39--79.

\bibitem{holtz2012manifolds}
{\sc S.~Holtz, T.~Rohwedder, and R.~Schneider}, {\em On manifolds of tensors of fixed {TT-rank}}, Numerische Mathematik, 120 (2012), pp.~701--731.

\bibitem{jacobsson2025homogeneous}
{\sc S.~Jacobsson}, {\em A homogeneous geometry of low-rank tensors}, arXiv preprint arXiv:2512.13594,  (2025).

\bibitem{kasai2016low}
{\sc H.~Kasai and B.~Mishra}, {\em Low-rank tensor completion: a {Riemannian} manifold preconditioning approach}, in Proceedings of The 33rd International Conference on Machine Learning, M.~F. Balcan and K.~Q. Weinberger, eds., vol.~48 of Proceedings of Machine Learning Research, New York, New York, USA, 6 2016, PMLR, pp.~1012--1021.

\bibitem{koch2010dynamical}
{\sc O.~Koch and C.~Lubich}, {\em Dynamical tensor approximation}, SIAM Journal on Matrix Analysis and Applications, 31 (2010), pp.~2360--2375, \url{https://doi.org/10.1137/09076578X}.

\bibitem{kolda2009tensor}
{\sc T.~G. Kolda and B.~W. Bader}, {\em Tensor decompositions and applications}, SIAM review, 51 (2009), pp.~455--500, \url{https://doi.org/10.1137/07070111X}.

\bibitem{kressner2014low}
{\sc D.~Kressner, M.~Steinlechner, and B.~Vandereycken}, {\em Low-rank tensor completion by {Riemannian} optimization}, BIT Numerical Mathematics, 54 (2014), pp.~447--468, \url{https://doi.org/10.1007/s10543-013-0455-z}.

\bibitem{molnar2018normal}
{\sc A.~Molnar, J.~Garre-Rubio, D.~P{\'e}rez-Garc{\'\i}a, N.~Schuch, and J.~I. Cirac}, {\em Normal projected entangled pair states generating the same state}, New Journal of Physics, 20 (2018), p.~113017, \url{https://doi.org/10.1088/1367-2630/aae9fa}.

\bibitem{oseledets2011tensor}
{\sc I.~V. Oseledets}, {\em Tensor-train decomposition}, SIAM Journal on Scientific Computing, 33 (2011), pp.~2295--2317, \url{https://doi.org/10.1137/090752286}.

\bibitem{pippan2010efficient}
{\sc P.~Pippan, S.~R. White, and H.~G. Evertz}, {\em Efficient matrix-product state method for periodic boundary conditions}, Physical Review B—Condensed Matter and Materials Physics, 81 (2010), p.~081103, \url{https://doi.org/10.1103/PhysRevB.81.081103}.

\bibitem{swijsen2022tensor}
{\sc L.~Swijsen, J.~Van~der Veken, and N.~Vannieuwenhoven}, {\em Tensor completion using geodesics on {Segre} manifolds}, Numerical Linear Algebra with Applications, 29 (2022), p.~e2446, \url{https://doi.org/10.1002/nla.2446}.

\bibitem{uschmajew2013geometry}
{\sc A.~Uschmajew and B.~Vandereycken}, {\em The geometry of algorithms using hierarchical tensors}, Linear Algebra and its Applications, 439 (2013), pp.~133--166.

\bibitem{uschmajew2020geometric}
{\sc A.~Uschmajew and B.~Vandereycken}, {\em Geometric methods on low-rank matrix and tensor manifolds}, in Handbook of variational methods for nonlinear geometric data, Springer, 2020, pp.~261--313.

\bibitem{vasilescu2003multilinear}
{\sc M.~A.~O. Vasilescu and D.~Terzopoulos}, {\em Multilinear subspace analysis of image ensembles}, in 2003 IEEE Computer Society Conference on Computer Vision and Pattern Recognition, 2003. Proceedings., vol.~2, IEEE, 2003, pp.~II--93, \url{https://doi.org/10.1109/CVPR.2003.1211457}.

\bibitem{verstraete2004density}
{\sc F.~Verstraete, D.~Porras, and J.~I. Cirac}, {\em Density matrix renormalization group and periodic boundary conditions: A quantum information perspective}, Physical review letters, 93 (2004), p.~227205.

\bibitem{ye2018tensor}
{\sc K.~Ye and L.-H. Lim}, {\em Tensor network ranks}, arXiv preprint arXiv:1801.02662,  (2018).

\bibitem{zhao2019learning}
{\sc Q.~Zhao, M.~Sugiyama, L.~Yuan, and A.~Cichocki}, {\em Learning efficient tensor representations with ring-structured networks}, in ICASSP 2019-2019 IEEE international conference on acoustics, speech and signal processing (ICASSP), IEEE, 2019, pp.~8608--8612, \url{https://doi.org/10.1109/ICASSP.2019.8682231}.

\bibitem{zhao2016tensor}
{\sc Q.~Zhao, G.~Zhou, S.~Xie, L.~Zhang, and A.~Cichocki}, {\em Tensor ring decomposition}, arXiv preprint arXiv:1606.05535,  (2016).

\end{thebibliography}
\end{document}